\documentclass[a4paper,12pt]{amsart}
\usepackage[english]{babel}
\usepackage[T1]{fontenc}
\usepackage{times}
\usepackage{microtype}
\usepackage{color}
\usepackage[normalem]{ulem}

\usepackage{commath,amssymb,amscd,mathrsfs,mathtools,dsfont,bbm,multirow,array,caption,comment,pifont}
\usepackage[all]{xy}
\usepackage{enumitem}
\usepackage{longtable}

\usepackage[noadjust]{cite}		

\allowdisplaybreaks

\usepackage[dvipsnames,svgnames,table]{xcolor}
\definecolor{blue}{RGB}{255,25,25}
\definecolor{blue}{RGB}{25,50,200}
\usepackage{tikz-cd}
\usepackage[pagebackref,linktocpage]{hyperref}

\hypersetup{
pdftitle={},				
pdfauthor={},	
colorlinks=true,			
linkcolor=red,			
citecolor=MidnightBlue,	
filecolor=blue,		
urlcolor=MidnightBlue	
}

\usepackage{cleveref}	

\theoremstyle{plain}
\newtheorem{thm}{Theorem}[section]
\newtheorem{theorem}[thm]{Theorem}

\newtheorem{lemma}[thm]{Lemma}
\newtheorem{corollary}[thm]{Corollary}
\newtheorem{proposition}[thm]{Proposition}
\theoremstyle{definition}
\newtheorem{remark}[thm]{Remark}

\newtheorem{definition}[thm]{Definition}

\newtheorem{example}[thm]{Example}

\numberwithin{equation}{section}

\newcommand{\p}{\partial}

\newcommand{\sA}{{\mathcal A}}

\newcommand{\sL}{{\mathcal L}}
\newcommand{\sM}{{\mathcal M}}

\newcommand{\sO}{{\mathcal O}}

\newcommand{\sV}{{\mathcal V}}

\newcommand{\C}{{\mathbb C}}

\newcommand{\BL}{{\mathbb L}}

\newcommand{\BP}{{\mathbb P}}

\newcommand{\R}{{\mathbb R}}

\newcommand{\W}{{\mathbb W}}

\newcommand{\Z}{{\mathbb Z}}

\newcommand{\h}{\widehat}
\newcommand{\w}{\widetilde}

\newcommand\Aut{\rm Aut}

\def\Jac{\mathop{\rm Jac}\nolimits}

\title[Holomorphic symplectic geometry of elliptic surfaces]{Holomorphic symplectic geometry of \\ elliptic surfaces}

\author{Jun-Muk Hwang and Guolei Zhong}

\thanks{This work was supported by the Institute for Basic Science (IBS-R032-D1). The second author was also supported by National Natural Science Foundation of China, and Science and Technology Commission of Shanghai Municipality (No. 22DZ2229014).}

\begin{document}

\begin{abstract}
 When a complex surface $X$ admits a nowhere vanishing holomorphic 2-form, it determines a (holomorphic)  symplectic structure on $X$. We study the symplectic geometry of such a symplectic structure when $X$ is an elliptic surface. When the elliptic fibration is nonisotrivial, we define a factorization of Kodaira's functional invariant, called the symplecto-functional invariant and prove that the symplecto-functional invariant determines the symplectic geometry of a nonisotrivial elliptic fibration. This leads to a classification of isogenies of nonisotrivial symplectic elliptic fibrations with a fixed source. We also classify isogenies of symplectic elliptic fibrations with a fixed target by studying symplectic automorphisms of germs of singular fibers.
As an application, we prove that a symplecto-biholomorphic map between germs of fibers of nonisotrivial elliptic K3 surfaces can be extended to compositions of isogenies of K3 surfaces.
\end{abstract}

\maketitle



{\bf Conventions}
\begin{itemize}
\item
We work in the complex-analytic category: all varieties are defined over the complex numbers, and all maps, bundles, and sections are holomorphic.
In particular, a symplectic form means a everywhere-nondegenerate holomorphic 2-form.
\item
An open subset of a complex manifold refers to its  Euclidean topology, whereas a Zariski-open subset means the complement of a closed analytic subset.
\item Let $F\colon C \to \C$ be a holomorphic function on a Riemann surface $C$.
The vanishing order of $F$ at a point $t \in C$ is denoted by ${\rm ord}_t F$.
For a holomorphic map $G\colon C \to C'$ between Riemann surfaces, we denote by ${\rm mult}_t G$ the local degree of $G$ at $t \in C$,
that is, if $z$ is a local coordinate on $C'$ centered at $G(t)$, then ${\rm mult}_t G = {\rm ord}_t(z \circ G)$.
We set ${\rm ord}_t F = \infty$ (resp. ${\rm mult}_t G = \infty$) if $F \equiv 0$ (resp. $G$ is a constant map).
\end{itemize}

\section{Introduction}
Let \(X\) be a (not necessarily compact) complex surface whose canonical bundle \(K_X\) admits a nowhere vanishing holomorphic section \(\omega \in H^0(X, K_X)\).
The section \(\omega\) defines a holomorphic symplectic form on \(X\).
In this paper, we study the symplectic geometry of such surfaces when there is an elliptic fibration \(f\colon X \to C\) over a Riemann surface \(C\), and we explore how the symplectic structure interacts with the underlying complex geometry.
Naturally, the most compelling case is when \(X\) is compact, particularly when it is an elliptic K3 surface. Nonetheless, the problem remains interesting on noncompact surfaces as well.

Let us start with some basic terminology.

\begin{definition}\label{d.basic}
Let $f\colon  X \to C$ be a proper surjective holomorphic map from a two-dimensional complex manifold $X$ onto a Riemann surface $C$.
\begin{itemize}
\item[(i)] We say that $f$ is {\em relatively minimal}, if there exists no $(-1)$-curve  in a fiber of $f$.
\item[(ii)]  We say that $f$ has {\em no multiple fiber}, if each  fiber of $f$ has a nonsingular point.
\item[(iii)] We say that $f$ is a {\em genus 1 fibration} if smooth fibers of $f$ have genus 1.
\item[(iv)]   In (iii), if   a section $\Sigma \subset X$ of $f$  is given, we call the pair $(f\colon  X \to C, \Sigma)$  an {\em elliptic fibration}. (Here we follow the convention of \cite{SS}: it is called an elliptic fibration with a section in \cite{FM} and \cite{Mi}.)
\item[(v)]  In (iii) (resp. (iv)), if a symplectic form $\omega$ on $X$ is given, we call the pair $(f\colon X \to C, \omega)$ (resp. the triple $(f\colon  X \to C, \Sigma, \omega)$)  {\em a symplectic genus 1 fibration} (resp. {\em a symplectic elliptic fibration}).
\item[(vi)]   For an unramified holomorphic map $\phi\colon U \to C$, the $\phi$-{\em pullback} $\phi^*(f\colon  X \to C)$  is the natural genus 1 fibration over $U$   obtained by the fiber product of $f$ and $\phi$. If $\Sigma \subset X$ is given as in (iv) (resp.  $\omega$ is given as in (v),   resp. $\Sigma \subset X$ and $\omega$ are given as in (iv) and (v)), it induces a natural elliptic fibration $\phi^*(f\colon  X \to C, \Sigma)$ (resp. a natural symplectic genus 1 fibration $\phi^*(f\colon  X \to C, \omega)$,  resp. a natural symplectic elliptic fibration $\phi^*(f\colon  X \to C, \Sigma,\omega)$).
\item[(vii)] In (vi), if $\phi$ is an inclusion of an open subset $O \subset C$, we write $$\phi^*(-) = (-)|_O,$$
where \((-)\)  stands for \((f\colon  X \to C)\), \((f\colon  X \to C,\Sigma)\), \((f\colon  X \to C,\omega)\), or \((f\colon  X \to C,\Sigma,\omega)\), respectively.
\end{itemize}
  \end{definition}

We have the following notion of isomorphisms.

\begin{definition}\label{d.isom}
Let $(f\colon  X \to C)$  and $(\w{f}\colon  \w{X} \to \w{C})$ be two genus 1 fibrations.
\begin{itemize}
\item[(i)]  A pair of biholomorphic maps $\phi\colon \w{C} \to C$ and $\Phi\colon \w{X} \to X$ is an {\em isomorphism of the two genus 1 fibrations} if $\phi \circ \w{f} = f \circ \Phi$. When \(C=\w{C}\) and \(\phi={\rm Id}_C\), we say that \(\Phi\) is an {\em isomorphism of the two genus 1 fibrations over \(C\)}.
\item[(ii)] Assume that sections $\Sigma\subset X$  and $\w{\Sigma}\subset\w{X}$ are given.
Then a pair $(\phi, \Phi)$ in (i) is an {\em isomorphism of the two elliptic fibrations} if  $ \Phi(\w{\Sigma}) = \Sigma$.
\[\xymatrix{
\w{\Sigma}\ar@{^(->}[r]\ar[d]_{\Phi|_{\Sigma}}&\w{X}\ar[d]_{\Phi}\ar[r]^{\w{f}}&\w{C}\ar[d]_{\phi}\\
\Sigma\ar@{^(->}[r]&X\ar[r]^f&C.
}\]
\item[(iii)] Assume that symplectic forms $\omega$ on $X$ and $\w{\omega}$ on $\w{X}$ are  given.  Then a pair $(\phi, \Phi)$ in (i) (resp. in (ii)) is an {\em isomorphism of the two symplectic genus 1 fibrations} (resp. {\em isomorphism of the two symplectic elliptic fibrations}) if $\Phi^* \omega = \w{\omega}.$ \end{itemize} 
\end{definition}

 Let \(\overline{\sM}=\overline{\mathcal{M}}_1\) denote the compactified moduli space of genus 1 curves.
We identify \(\overline{\mathcal{M}}\) with the projective line \(\mathbb{P}^1 = \mathbb{C} \cup \{\infty\}\) via the normalized modular invariant \(\frac{1}{1728} \cdot \jmath\), where \(\jmath\) denotes the classical elliptic modular function.

\begin{definition}\label{d.functional}
Let $f\colon  X \to C$ be a genus 1 fibration.
There are two invariants associated to $f\colon  X \to C$ (see \cite[Section 7]{Ko}, \cite[pages 202, 211]{BHPV}):
\begin{itemize}
\item[(i)] {\em the functional invariant}, a holomorphic map $J\colon  C  \to \overline{\sM}= \BP^1$ given by the moduli of fibers of $f$; and
\item[(ii)] the {\em homological invariant}, a group homomorphism $$ \pi_1( C^{\rm reg}, c_0) \to {\Aut}(H_1(f^{-1}(c_0), \Z)) \cong {\rm SL}_2(\Z)$$ given by the monodromy of the smooth  fibration $f^{-1}(C^{\rm reg}) \to C^{\rm reg}$ at a base point $c_0 \in C^{\rm reg}$, where $C^{\rm reg} \subset C$ is the Zariski-open subset of regular values of $f$.
\end{itemize}
We say that $f$ is {\em isotrivial} (resp. {\em nonisotrivial}) if $J$ is a constant (resp. nonconstant) map.
\end{definition}

Recall that associated to a genus 1 fibration $f\colon X \to C$ without multiple fibers is its Jacobian fibration
$\Jac(f)\colon  \Jac (X) \to C$, which is a relatively minimal elliptic fibration, with the same functional and homological invariants as $f\colon  X \to C$ (see \cite[Section V.9]{BHPV}).

Our first result compares the symplectic geometry of $f\colon  X \to C$ with that of $\Jac(f)\colon  \Jac(X) \to C$  as follows.

\begin{theorem}\label{t.Jac}
Let $f\colon  X \to C$ be a genus 1 fibration with no multiple fibers and let $\Jac(f)\colon  \Jac(X) \to C$ be its Jacobian fibration.
\begin{itemize}
\item[(i)]
If there exists a symplectic form $\omega$ on $X$, then there exists a symplectic form $\Jac(\omega)$ on $\Jac(X)$ such that each $y \in C$ admits a neighborhood $y \in O \subset C$ with a biholomorphic map $\varphi\colon  f^{-1}(O) \to \Jac(f)^{-1}(O)$,  which is an isomorphism of the two symplectic genus 1 fibrations $(f \colon X \to C, \omega)|_O$ and $(\Jac(f) \colon \Jac(X) \to C, \Jac(\omega))|_O$.
\item[(ii)] Conversely, if there exists a symplectic form $\Jac(\omega)$ on $\Jac(X)$, then there exists a symplectic form $\omega$ on $X$ such that each $y \in C$ admits a neighborhood $y \in O \subset C$ with a biholomorphic map $\varphi\colon  f^{-1}(O) \to \Jac(f)^{-1}(O)$, which is an isomorphism of the two symplectic genus 1 fibrations $(f \colon X \to C, \omega)|_O$ and $(\Jac(f) \colon \Jac(X) \to C, \Jac(\omega))|_O$.
\end{itemize} \end{theorem}

Theorem \ref{t.Jac} is rather easy to prove, and it gives alternative proofs of results like  \cite[Proposition 5.4, Chapter 11]{Hu}.
Theorem \ref{t.Jac} suggests that to study the symplectic geometry of genus 1 fibrations,  we may concentrate on the case of elliptic fibrations.
We refer the reader to \cite[Theorem 3.6]{AR25} for a related result in the higher-dimensional case.

In our study of symplectic elliptic fibrations,
the following notion is crucial.
It is motivated by isogenies of K3 surfaces in \cite[Section 12.2.5]{SS}.

\begin{definition}\label{d.isogeny}
Let $(f\colon  X \to C, \Sigma, \omega)$ be a symplectic elliptic fibration.
An {\em isogeny} from a symplectic elliptic fibration $(\w{f}\colon  \w{X} \to \w{C}, \w{\Sigma}, \w{\omega})$  to  $(f\colon  X \to C, \Sigma, \omega)$ is a pair $(\psi,  \Psi)$ of
  \begin{itemize} \item[(i)] a surjective holomorphic map $\psi\colon  \w{C} \to C$, and
  \item[(ii)] a dominant meromorphic map $\Psi\colon \w{X} \dasharrow X$, \end{itemize}  such that
there exists a nonempty Zariski-open subset $U \subset \w{C}$ satisfying the following conditions:
\begin{itemize} \item[(a)] $\psi$ is unramified on $U$; \item[(b)]  the meromorphic map $\Psi$ is holomorphic on $\w{f}^{-1}(U)$;  and \item[(c)]  the holomorphic map $\Psi|_{\w{f}^{-1}(U)}$ in (b)  induces an isomorphism of symplectic elliptic fibrations $$ (\w{f}\colon  \w{X} \to  \w{C}, \w{\Sigma}, \w{\omega})|_U \stackrel{\cong}{\rightarrow} (\psi|_U)^*(f\colon  X \to C, \Sigma, \omega).$$ \end{itemize}  \end{definition}

The following is a typical example on the isogeny of symplectic elliptic fibrations.
\begin{example}[Shioda]\label{ex:Shi1}
Let \(X\) be an elliptic K3 surface over $P^1_T$, the projective line with an inhomogeneous coordinate $T$,  with two singular fibers of type \(\textup{II}^*\).
By \cite[Proposition 2.1]{Shio}, such \(X\) has a  Weierstrass equation 
\[
X_{\alpha,\beta}\colon y^2=x^3-3\alpha x+(T+\frac{1}{T}-2\beta).
\]
Let \(S\) be a Kummer surface admitting an elliptic fibration \(S\to\mathbb{P}_t^1\) which has two singular fibers of type \(\textup{IV}
^*\) and has a defining equation (cf. \cite[Proposition 3.1]{Shio})
\[
y^2=x^3-3\alpha x+(t^2+\frac{1}{t^2}-2\beta).
\]
Then the rational map \(\Phi\colon S\dashrightarrow X\) sending \((x,y,t)\) to \((x,y,t^2)\) defines an isogeny with respect to the symplectic elliptic fibrations \(S\to\mathbb{P}_t^1\) and \(X\to\mathbb{P}_T^1\). 
\end{example}

Our second main result gives a sort of a classification of isogenies with a given source when the elliptic fibration is nonisotrivial.

\begin{theorem}[{see Proposition \ref{p.reduction}}]\label{t.factorize}
Let $(f\colon  X \to C, \Sigma, \omega)$ be a nonisotrivial symplectic elliptic fibration. Then there exists an isogeny $(\xi\colon  C \to \h{C}, \Xi\colon  X \dasharrow \h{X})$ from $(f\colon  X \to C, \Sigma, \omega)$ to a symplectic elliptic fibration $(\h{f}\colon  \h{X} \to \h{C}, \h{\Sigma}, \h{\omega})$ with the following universal property: for any isogeny $(\psi\colon  C \to C', \Psi\colon X \dasharrow X')$ from $(f\colon  X \to C, \Sigma, \omega)$ to a symplectic elliptic fibration $(f'\colon  X' \to C', \Sigma', \omega')$,  there exists a unique isogeny $(\xi'\colon  C' \to \h{C}, \Xi'\colon  X' \dasharrow \h{X})$ from $(f'\colon  X' \to C', \Sigma', \omega')$ to $(\h{f}\colon  \h{X} \to \h{C}, \h{\Sigma}, \h{\omega})$ such that $$\xi= \xi' \circ \psi \ \mbox{ and } \ \Xi = \Xi' \circ \Psi.$$ In particular,  the symplectic elliptic fibration $(\h{f}\colon  \h{X} \to \h{C}, \h{\Sigma}, \h{\omega})$ and the isogeny $(\xi, \Xi)$ are uniquely determined by $(f\colon  X \to C, \Sigma, \omega)$.  \end{theorem}

\begin{remark}\label{r.Shioda}
In \cite[Theorem 1.1]{Shio}, besides Example \ref{ex:Shi1}, Shioda also constructed another isogeny $\Psi\colon X \dasharrow S$ which is also of degree 2.
By Theorem \ref{t.factorize}, the two maps $\Psi$ and $\Phi$ are isogenies with respect to different elliptic fibration structures of the Kummer surface  $S$.
\end{remark}

\begin{definition}\label{d.reduction}
The symplectic elliptic fibration $(\h{f}\colon  \h{X} \to \h{C}, \h{\Sigma}, \h{\omega})$ in Theorem \ref{t.factorize} is called the {\em symplecto-elliptic reduction} of $(f\colon  X \to C, \Sigma, \omega)$. We say that $(f\colon  X \to C, \Sigma, \omega)$ is {\em symplecto-elliptically reduced} if it is its own symplecto-elliptic reduction, namely, both $\xi$ and $\Xi$ are identity maps. 
\end{definition}

We give a few examples below to explain Definition \ref{d.reduction}.
We refer the reader to Example \ref{ex.Og} for a more interesting example on the symplecto-elliptic reduction.
\begin{example}\label{ex.reduced}
\begin{enumerate}
\item  If the functional
invariant of a non-isotrivial elliptic K3 surface \(X\to C\) does not allow nontrivial factorizations,
for example, if the functional invariant is of degree one, then \(X\to C\) is symplecto-elliptically reduced. 
\item By \cite[Table A, \(\mathscr{J}_5\)]{Og}, there exists a Kummer surface \(X\) with a Jacobian fibration \(X\to C\) whose singular fibers consist of one  \(\textup{I}_6^*\)-fiber and six \(\textup{I}_2\)-fibers.
Then \(X\to C\) is symplecto-elliptically reduced (see Theorem \ref{t.auto}).
\end{enumerate}
\end{example}

The proof of Theorem \ref{t.factorize}
 needs a new notion called the {\em symplecto-functional invariant} (see Definition \ref{d.invariant}).
Briefly speaking, the symplecto-functional invariant of a nonisotrivial symplectic elliptic fibration $(f\colon  X \to C, \Sigma, \omega)$ is a holomorphic map $\lambda\colon  C^o \to \sL$  from a Zariski-open subset $C^o \subset C$ (determined by $f$) to a natural $\C^*$-bundle $\varpi\colon  \sL \to \sM^{\star} \coloneqq \overline{\sM} \setminus \{ 0, 1, \infty \},$ which gives
 a factorization of  the functional invariant $J|_{C^o}\colon  C^o\subset C \to \overline{\sM}$:
\[
\xymatrix{
C\ar[d]^J&C^o\ar[dr]^{\lambda}\ar@{_(->}[l]&\\
\overline{\sM}&\sM^{\star}\ar@{_(->}[l]&\sL.\ar[l]_\varpi
}
\]

It turns out that the symplecto-functional invariant $\lambda$ determines the symplectic geometry of $(f\colon  X \to C, \Sigma, \omega)$ in the following sense.

\begin{theorem}[{cf.~Theorem \ref{t.lambda*}}]\label{t.lambda}
Let $(f\colon  X \to C, \Sigma, \omega)$ and $(\w{f}\colon  \w{X} \to \w{C}, \w{\Sigma}, \w{\omega})$ be two nonisotrivial symplectic elliptic fibrations.  Let $\lambda\colon  C^o \to \sL$ and $\w{\lambda}\colon  \w{C}^o \to \sL$ be their symplecto-functional invariants.  Then the following two statements are equivalent. \begin{itemize}
\item[(a)] There exists  a biholomorphic map $\phi\colon C \to \w{C}$ satisfying $\phi(C^o) = \w{C}^o $ and $ \lambda = \w{\lambda} \circ \phi|_{C^o}.$
     \item[(b)] There exists an isomorphism  between the symplectic elliptic fibrations  $(f\colon  X \to C, \Sigma, \omega)$ and $(\w{f}\colon  \w{X} \to \w{C}, \w{\Sigma}, \w{\omega})$.
     \end{itemize}
     In fact, the isomorphism in (b) can be chosen such that the corresponding biholomorphic map from $C$ to $\w{C}$ is $\phi$ in (a).
     \end{theorem}

Recall that an elliptic fibration is determined by its functional and homological invariants, but its functional invariant alone cannot determine it.
Theorem \ref{t.lambda} says that the symplecto-functional invariant, which is a factorization of the functional invariant, determines the elliptic fibration.
Thus the information of the homological invariant is included in the factorization.

Based on Theorem \ref{t.lambda}, we prove Theorem \ref{t.factorize} by constructing the symplecto-elliptic reduction $(\h{f}\colon  \h{X} \to \h{C}, \h{\Sigma}, \h{\omega})$ using the image of the symplecto-functional invariant in $\sL$.

We have the following application of  Theorem \ref{t.factorize} to elliptic K3 surfaces.

\begin{theorem}\label{t.K3}
Let $(f_1\colon  X_1\to C_1, \Sigma_1, \omega_1)$
and $(f_2\colon  X_2 \to C_2, \Sigma_2, \omega_2)$ be two nonisotrivial symplectic elliptic fibrations of K3 surfaces. 
Assume that there are  connected open subsets $O_1 \subset C_1$ and $O_2 \subset C_2$ and an isomorphism of the symplectic elliptic fibrations $(f_1\colon  X_1\to C_1, \Sigma_1, \omega_1)|_{O_1}$ and $(f_2\colon  X_2 \to C_2, \Sigma_2, \omega_2)|_{O_2}$ given by a pair  of biholomorphic maps $\phi\colon O_1 \to O_2$ and $\Phi\colon f_1^{-1}(O_1) \to f_2^{-1}(O_2)$.
\begin{itemize}
\item[(i)]  Then the symplecto-elliptic reduction of $(f_1\colon  X_1\to C_1, \Sigma_1, \omega_1)$ is naturally isomorphic to the symplecto-elliptic reduction of $(f_2\colon  X_2 \to C_2, \Sigma_2, \omega_2)$.
 \item[(ii)] The isomorphism $(\phi, \Phi)$ is compatible with the symplecto-elliptic reductions, as follows. 
\[\xymatrix@R=1.5em{
&f_1^{-1}(O_1)\ar[r]^\Phi_\sim\ar@{^{(}->}[dl]&f_2^{-1}(O_2)\ar@{_{(}->}[dr]&\\
X_1\ar@{-->}[r]^{\Xi_1}\ar[d]_{f_1}&\h{X}_1\ar@{=}[r]\ar[d]&\h{X}_2\ar[d]&X_2\ar@{-->}[l]_{\Xi_2}\ar[d]^{f_2}\\
C_1\ar[r]^{\xi_1}&\h{C}_1\ar@{=}[r]&\h{C}_2&C_2\ar[l]_{\xi_2}\\
&O_1\ar@{_{(}->}[ul]\ar[r]^\phi_\sim&O_2\ar@{^{(}->}[ur]&
}\]
\item[(iii)]
If, furthermore, both $(f_1\colon  X_1\to C_1, \Sigma_1, \omega_1)$ and $(f_2\colon  X_2 \to C_2, \Sigma_2, \omega_2)$ are symplecto-elliptically reduced, then $\Phi$ can be extended to a symplectic biholomorphic map between $X_1$ and $X_2$.
\end{itemize}
\end{theorem}

In the opposite direction of Theorem \ref{t.factorize},
we also give the following classification of isogenies with a given target, the statement of which requires Theorem \ref{t.quotient} on symplectic automorphisms of germs of symplectic elliptic fibrations.

\begin{theorem}\label{t.isogeny} Let $(f\colon  X \to C, \Sigma, \omega)$  be a symplectic elliptic fibration.  \begin{itemize} \item[(i)] For an isogeny $(\psi, \Psi)$ as in Definition \ref{d.isogeny} and a point $ y\in \w{C}$ where $\psi$ is ramified, the equivalence class of the germ of $f$ at $\psi(y)$  and $d\coloneqq {\rm mult}_y \psi$ must be one of the list in Theorem \ref{t.quotient}. 
\item[(ii)] Conversely, let $\psi\colon  \w{C} \to C$ be a surjective holomorphic map from a Riemann surface $\w{C}$ such that for any point $y \in \w{C}$ where $\psi$ is ramified, the equivalence class of the germ of $f$ at $\psi(y)$  and $d\coloneqq {\rm mult}_y \psi$ is one of the list in Theorem \ref{t.quotient}. Then there exists a symplectic elliptic fibration $(\w{f}\colon  \w{X} \to \w{C}, \w{\Sigma}, \w{\omega})$ and  \(\psi\) extends to an isogeny $(\psi, \Psi)$ as in Definition \ref{d.isogeny}. Furthermore, both $(\w{f}\colon  \w{X} \to \w{C}, \w{\Sigma}, \w{\omega})$ and $\Psi$ are uniquely determined (up to isomorphisms) by $\psi$.\end{itemize}
\end{theorem}

 From Theorem \ref{t.K3} (iii), it is interesting to know which elliptic K3 surfaces are symplecto-elliptically reduced. 
 Besides Example \ref{ex.reduced}, a more interesting example is the following, the proof of which uses Theorem \ref{t.isogeny}.

\begin{theorem}\label{t.generic}
Let $X$ be a K3 surface with Picard number $\rho(X)\leq 9$ which admits a non-isotrivial symplectic elliptic fibration $(f\colon  X \to C, \Sigma, \omega)$. Then it is symplecto-elliptically reduced.
\end{theorem}

The conditions of Theorem \ref{t.generic} hold for a  very general elliptic K3 surface (see  \cite[Remark 4.2]{Gee}).
On the other hand,  we have learned from Keiji Oguiso the following explicit example of an elliptic K3 surface \(\w{X} \to \w{C}\)  which is not symplecto-elliptically reduced.

\begin{example}\label{ex.Og} 
By \cite[Table A, \(\mathscr{J}_7\)]{Og},
there exists a nonisotrivial elliptic K3 surface \(X \to C \cong\mathbb{P}^1\) with exactly five singular fibers: two  \(\textup{I}_0^*\)-fibers, two \(\textup{I}_1\)-fibers, and one \(\textup{I}_4^*\)-fiber.
Considering that there is only one \(\textup{I}_4^*\)-fiber, we deduce that \(X\to C\) is symplecto-elliptically reduced (see Theorem \ref{t.auto}). 

Let \(p_1,p_2\in C\) be the two points over which the fibers are of type \(\textup{I}_0^*\) and
take a double cover \(\mathbb{P}^1 \cong \w{C} \to C\) which is  ramified exactly at \(p_1\) and \(p_2\). 
Then Theorem \ref{t.isogeny} (corresponding to Theorem \ref{t.quotient} (1), (2), or (5)) gives an elliptic K3 surface $\w{X} \to \w{C}$  and an isogeny \(\w{X} \dasharrow X\) under  a suitable choice of symplectic forms on $X$ and $\w{X}$.
The singular fibers of \(\w{X} \to \w{C}\) consist of two \(\textup{I}_4^*\)-fibers and four \(\textup{I}_1\)-fibers.
In particular, \(X\to C\) is the symplecto-elliptic reduction of \(\w X\to\w C\).
\end{example}

\begin{remark}
By Shioda–Tate Formula (see \cite[Chapter 11, Corollary 3.4]{Hu}), the K3 surface in Example \ref{ex.Og} has Picard number 18. 
We have learned from Philip Engel that the moduli of elliptic K3 surfaces having two fibers of type $\textup{I}_0^*$ has dimension 10 and a general such K3 surface has Picard number 10. 
Taking a double covering as in Example \ref{ex.Og}  of such a K3 surface, we obtain an elliptic K3 surface with Picard number 10 which is not symplecto-elliptically reduced (see \cite[Theorem 12.9]{SS}). 
Hence, the condition $\rho(X) \leq 9$ in Theorem \ref{t.generic} is optimal. 
\end{remark}

Note that the assumption of nonisotriviality in Theorem \ref{t.K3} is crucial.
In contrast, we prove the following theorem in the isotrivial case, which says that the symplectic geometry on the germs of neighborhoods of fibers is completely determined by the underlying complex geometry.

\begin{theorem}\label{t.isotrivial}
Let $(f\colon X \to C, \Sigma, \omega)$ and $(\w{f}\colon  \w{X} \to \w{C}, \w{\Sigma}, \w{\omega})$  be two isotrivial symplectic elliptic fibrations. 
Assume that there is a point $y \in C$ (resp. $\w{y} \in \w{C}$) with a neighborhood $y \in O \subset C$ (resp. $ \w{y} \in \widetilde{O} \subset \w{C}$) such that the two  elliptic fibrations $(f\colon X \to C, \Sigma)|_O$ and $(\w{f}\colon  \w{X} \to \w{C}, \w{\Sigma})|_{\w{O}}$ are isomorphic. Then there exists a neighborhood $y \in U \subset O$ (resp. $\w{y} \in \w{U} \subset \w{O}$) such that the two symplectic elliptic fibrations  $(f\colon X \to C, \Sigma, \omega)|_U$ and $(\w{f}\colon  \w{X} \to \w{C}, \w{\Sigma}, \w{\omega})|_{\w{U}}$ are isomorphic. \end{theorem}

\medskip
Throughout the paper, we use the classical theory of elliptic surfaces originated from Kodaira \cite{Ko} and developed by many authors.  We cite results from standard references  \cite{BHPV}, \cite{FM} \cite{Hu}, \cite{Mi} and \cite{SS}. In addition, our methodology has two  ingredients from symplectic geometry: action-angle variables and symplectic singularities.  A symplectic elliptic fibration is a  Lagrangian fibration, equivalently, a completely integrable Hamiltonian system. The notion of action-angle variables in the theory of completely integrable Hamiltonian systems (Definition \ref{d.chi} and Theorem \ref{t.AA}) plays a key role in a number of steps in our arguments.
On the other hand, in dealing with  the local geometry of a  ramification point of an isogeny between symplectic elliptic fibrations, we use ideas (Definition \ref{d.singular} and Proposition \ref{p.Beauville})
from the theory of symplectic singularities originated from Beauville (\cite{Be},\cite{Fu}).

\subsection*{Outline.}
The outline of the paper is as follows.
In Section \ref{s.Jac}, we apply a  version of the action-angle variables to prove Theorem \ref{t.Jac}.
In Sections \ref{s.invariant} and \ref{s.reduction}, we focus on the symplectic geometry of non-isotrivial elliptic fibrations: we introduce the symplecto-functional invariant and prove Theorem \ref{t.lambda} (cf.~Theorem \ref{t.lambda*}) in Section \ref{s.invariant}, and we construct the symplecto-elliptic reduction and prove Theorem \ref{t.K3} in Section \ref{s.reduction}.
Then in Section \ref{s.auto}, we classify the symplectic automorphisms of germs of symplectic elliptic fibrations, which are used to prove Theorem \ref{t.isogeny} and Theorem \ref{t.generic} in Section \ref{s.quotient}.
Finally, we study the symplectic geometry of germs of isotrivial elliptic fibrations and prove Theorem \ref{t.isotrivial} in Section \ref{s.isotrivial}.

\subsection*{Acknowledgement}
The authors would like to thank Keiji Oguiso for valuable suggestions including Example \ref{ex.Og}, and  Mathias Sch\"utt for helpful comments on the first draft of the paper, especially, drawing our attention to \cite{Shio} and suggesting the current improved version of Theorem \ref{t.generic}.
We are grateful to Philip Engel for telling us his observation that $\rho(X) \leq 9$ in Theorem \ref{t.generic} is optimal and also to the referee for helpful suggestions to improve the presentation.

\section{Symplectic forms on genus 1 fibrations and their Jacobian fibrations}\label{s.Jac}

We recall the basic structure theory of smooth symplectic elliptic fibrations, a holomorphic version of the action-angle variables in completely integrable Hamiltonian systems (for example, see Theorem 44.1  of \cite{GS}).

\begin{definition}\label{d.chi}
Let $(f\colon  X \to C, \Sigma, \omega)$ be a smooth symplectic elliptic fibration. Then there is a canonical unramified covering map $\chi\colon  T^* C\to X$ which commutes with the projections to $C$, defined as follows.

For each $y \in C$, let $X_y = f^{-1}(y)$  be the fiber of $f.$ For  each element  $v \in T^*_y C,$ we have a holomorphic section $f^*v \in H^0(X_y, T^*X|_{X_y})$, which determines a unique holomorphic section  $$\vec{v} \ \in \ H^0(X_y, TX_y) \ \subset \ H^0(X_y, TX|_{X_y})$$ satisfying $\omega(\vec{v}, \cdot) = f^* v.$ Integrating the vector field $\vec{v}$ for each $v \in T^*_y C$,  we obtain an action of $T^*_y C$ on $X_y$ and  the orbit map $\chi_y\colon  T_y^*C \to X_y$ of $\Sigma \cap X_y$. Varying $y \in C$, we obtain the unramified covering map $\chi\colon  T^* C\to X$. \end{definition}

The following is \cite[Theorem 44.2]{GS}. In \cite{GS}, it is formulated in the $C^{\infty}$-setting, but the argument works in the holomorphic setting as well.

\begin{theorem}\label{t.AA}
Let $(f\colon  X \to C, \Sigma, \omega)$ and $\chi\colon  T^*C \to X$ be as in Definition \ref{d.chi}.
Let $\omega^C$ be the standard symplectic form on $T^*C$.
Then $\chi^* \omega = \omega^C.$ \end{theorem}

Recall the following elementary fact.

\begin{lemma}\label{l.standard}
Let $C$ be a Riemann surface. For any section $\eta \in H^0(C, T^*C)$, let   $\tau_{\eta}\colon  T^*C \to T^*C$ be the translation by $\eta$, namely, the biholomorphic automorphism sending $x \in T_y^*C$ to $\tau_{\eta}(x) = x + \eta(y).$
  Then the standard symplectic form $\omega^C$ on $T^*C$ satisfies $\tau_{\eta}^* \omega^C = \omega^C$, namely, it is invariant under $\tau_{\eta}$.
\end{lemma}

\begin{proof}
Recall   (see,  for example, page 220 in \cite{GS}) that for any holomorphic coordinate $z$ on any coordinate neighborhood $O \subset C$ and the holomorphic function $w$  on $T^*O$ given by  $\frac{\p}{\p z},$ $$ \omega^C|_{\pi^{-1}(O)} = {\rm d} z \wedge {\rm d} w.$$
In terms of the coordinates $(z,w)$, the section $\eta|_O$ is given by $w = g(z)$ for some holomorphic function $g(z)$ and
$\tau_{\eta} (z,w) = (z, w + g(z)).$ Then  $$\tau_{\eta}^* \omega^C = {\rm d} z \wedge {\rm d}(w + g(z)) = {\rm d} z \wedge {\rm d} w = \omega^C$$ on $T^*O$, proving the lemma. \end{proof}

The following is a direct corollary of Theorem  \ref{t.AA} and Lemma \ref{l.standard}.

\begin{lemma}\label{l.trans}
In Theorem  \ref{t.AA}, let  $\tau_{\eta}\colon  X \to X$ be the translation by a section of $f$, which is a biholomorphic automorphism commuting with $f$.  Then $\tau_{\eta}^* \omega = \omega.$
\end{lemma}

Theorem \ref{t.Jac} can be proved easily using Lemma \ref{l.trans} as follows.

\begin{proof}[Proof of Theorem \ref{t.Jac}]
Let us recall the relation (the notion of the analytic Tate-Shafarevich group) between a genus 1 
fibration $f\colon  X \to C$ and its Jacobian fibration $\Jac(f)\colon  \Jac(X) \to C$, from \cite[Theorem 11.1, Chapter V]{BHPV} or \cite[Proposition 5.5, Chapter I]{FM}.
 Let $\sA$ be the sheaf of local holomorphic sections of $\Jac(f)$, a sheaf of  abelian groups on  $C.$
 Given an element $\eta \in H^1(C, \sA),$ represent it by a 1-cocycle  $\{ \eta_{ij} \}$ with values in $\sA$ relative to an open cover $\{U_i \mid i \in I\}$ of $C$ such that $U_i \cap U_j \subset C^{\rm reg}$ for $i \neq j \in I$.
 Then the translation $$\tau_{ij}\colon  \Jac(f)^{-1}(U_i \cap U_j) \to \Jac(f)^{-1}(U_i \cap U_j)$$ given  by the section $\eta_{ij}$ of $\Jac(f)$ determines a new gluing of $\Jac(f)^{-1}(U_i)$ to $\Jac(f)^{-1}(U_j)$ to define a genus 1 
 fibration $f^{\eta}\colon  X^{\eta} \to C$ such that $\Jac(f^{\eta}) = \Jac(f)$. In fact, there is an element $\eta \in H^1(C, \sA)$ such that $f^{\eta}\colon  X^{\eta} \to C$ is isomorphic to $f\colon  X \to C$ as genus 1 fibrations over \(C\).

 Let $\{ X_i \mid i \in I\}$ be the open cover of $X = X^{\eta}$ corresponding to $\Jac(f)^{-1}(U_i)$ in the above gluing.
 Suppose there exists a symplectic form $\omega$ on $X = X^{\eta}$.
 Let $\omega_i$ be the symplectic form on $\Jac(f)^{-1}(U_i)$ corresponding to $\omega|_{X_i}$. Then $\omega_i|_{\Jac(f)^{-1}(U_i \cap U_j)} = \omega_j|_{\Jac(f)^{-1}(U_i \cap U_j)}$ by definition. By Lemma \ref{l.trans}, they are unchanged under the inverse translation $\tau_{ij}^{-1}$.
 Thus they induce a symplectic form $\Jac(\omega)$ on $\Jac(X)$. This proves Theorem \ref{t.Jac} (i): the map $\varphi$ is obtained from the biholomorphic map $X_i \cong \Jac(f)^{-1}(U_i)$ and is obviously an isomorphism of the symplectic genus 1 fibrations. A parallel argument proves (ii). \end{proof}

\section{Symplecto-functional invariant of a nonisotrivial symplectic elliptic fibration}\label{s.invariant}

\begin{definition}\label{d.lattice}
  Recall that $\overline{\sM}$ is the compactification of the moduli space $\sM$ of curves of genus 1, and we have identified it with $\BP^1$ such that $\sM$ is identified with $\C \subset \BP^1.$ Set $\sM^{\star} \coloneqq \sM \setminus\{ 0, 1\}$. Let $V$ be a 1-dimensional complex vector space. \begin{itemize}
\item[(i)] Let $\R{\rm Bases} (V)$ be the set of oriented $\R$-bases of $V$.  It is a complex manifold of dimension 2: an identification of  $V$ with $\C$ gives an identification
$$\R{\rm Bases}(\C) = \{ (z_1, z_2)  \mid z_1, z_2 \in \C \setminus \{0\}, {\rm Im}(z_2/z_1) >0\}.$$
\item[(ii)] The multiplicative group $\C^*$ acts on $\R{\rm Bases}(V)$ freely: an element $t \in \C^*$ sends $(v_1, v_2)$ to $(t v_1, t v_2)$.
\item[(iii)] The group ${\rm SL}_2(\Z)$ acts freely and discontinuously on $\R{\rm Bases}(V)$ by $$(v_1, v_2) \mapsto (a v_1 + b v_2, c v_1 + d v_2)$$ for $a,b,c,d \in \Z$ with $ad-bc =1$.
Each element $(v_1, v_2) \in \R{\rm Bases}(V)$ determines a lattice $\Z v_1 + \Z v_2 \subset V$.
Two elements determine the same lattice if and only if they are in the same ${\rm SL}_2(\Z)$-orbit.
The quotient complex manifold
$$\BL(V) \coloneqq \R{\rm Bases}(V)/ {\rm SL}_2(\Z)$$
is the {\em space of lattices} in $V$.
  \item[(iv)]    There is a canonical holomorphic map $\gamma_V\colon \BL(V) \to \sM$,  sending a lattice $\Z v_1 + \Z v_2$ in $V$ to the elliptic curve $V/(\Z v_1 + \Z v_2)$.
  \item[(v)] Two elliptic curves \(V/(\Z v_1 + \Z v_2)\) and \(V/(\Z v_3 + \Z v_4)\) are isomorphic if and only if their lattices are homothetic, i.e., there exists \(\alpha\in\mathbb{C}^*\) such that \((v_3,v_4)=(\alpha v_1,\alpha v_2)\).
    \item[(vi)] The $\C^*$-action in (ii) descends to a $\C^*$-action on $\BL(V)$ such that the isotropy subgroup at any point in $\gamma_V^{-1}(\sM^{\star})$ is $\{ \pm 1 \} \subset \C^*.$
   Hence \(\gamma_V^{-1}(\sM^{\star})\to \sM^{\star}\) is a \(\mathbb{C}^*/\{\pm 1\}\cong\mathbb{C}^*\)-bundle.
    \end{itemize}
    Let $\sV$ be a line bundle on a complex manifold $M$.
    \begin{itemize}
      \item[(vii)] Let $\pi_{\sV}\colon \BL(\sV) \to M$ be the fiber bundle whose fiber $\BL(\sV)_x$ at $x \in M$ is the space of lattices $\BL(\sV_x)$ in  the fiber $\sV_x$ of the line bundle $\sV$ at $x$.
      Denote by $\gamma_{\sV}\colon \BL(\sV) \to \sM$  the holomorphic map defined by $$\gamma_{\sV}|_{\BL(\sV)_x} = \gamma_{\sV_x} \mbox{  for each } x \in M.$$  \end{itemize} \end{definition}

\begin{definition}\label{d.sL}
In Definition \ref{d.lattice} (vii),
 we choose \(M\) to be $\sM^{\star}$ and choose $\sV$ to be 
the cotangent  bundle $T^*\sM^{\star}$. 
 Then we have  holomorphic maps
  \[
  \xymatrix{
  \BL(T^*\sM^{\star})\ar[rr]^{\gamma_{T^* \sM^{\star}}}\ar[d]_{\pi_{T^* \sM^{\star}}}&& \sM\ar[d]^{\jmath}\\
  \sM^{\star}\ar[rr]^{\jmath}&&\C.
  }
  \]
 Let $\sL \subset \BL(T^*\sM^{\star})$ be the hypersurface in $\BL(T^*\sM^{\star})$ defined by the equation $$\jmath \circ \gamma_{T^* \sM^{\star}} - \jmath \circ \pi_{T^* \sM^{\star}} \ = \ 0.$$ Let $\varpi\colon  \sL \to \sM^{\star}$ be the restriction of $\pi_{T^*\sM^{\star}}$ to $\sL$. We call it the {\em tautological lattice bundle}. It is a $\C^*/\{\pm 1\}$-fiber bundle by Definition \ref{d.lattice} (vi).
 \end{definition}

\begin{definition}[Symplecto-functional invariant]\label{d.invariant}
Let $(f\colon  X \to C, \Sigma, \omega)$ be a nonisotrivial symplectic elliptic fibration.
Let $J\colon  C \to \overline{\sM}$ be its functional invariant and set $$C^o \coloneqq \{ y \in C \mid J(y) \in \sM^{\star}, X_y \mbox{ is smooth},  \mbox{and}~J \mbox{ is unramified at } y\}.$$
By the flatness of \(f\), the restriction $(f\colon  X \to C, \Sigma, \omega)|_{C^o}$ is then a smooth elliptic fibration whose functional invariant is unramified.

For each $y \in C^o$, we have the unramified holomorphic map $\chi_y\colon T^*_yC \to X_y$ from Definition \ref{d.chi}. Then $\chi_y^{-1}(\Sigma \cap X_y)$ is a lattice in $T^*_y C$.  Since $J$ is unramified at $y$, we have the linear isomorphism $T^*_y C \cong T^*_{J(y)} \sM^{\star}$ induced by ${\rm d}_y J\colon  T_y C \to T_{J(y)} \sM^{\star}$, which sends the lattice  $\chi_y^{-1}(\Sigma \cap X_y) \subset T^*_y C$ to an  element $\lambda(y) \in \BL(T^*_{J(y)} \sM^{\star})$.
 Since the elliptic curve $X_y \cong T^*_y C/\chi_y^{-1}(\Sigma \cap X_y)$ has moduli type $J(y)$,  we have \(\gamma_{T^*\sM^{\star}}(\lambda(y))=J(y)\),  implying that  $\lambda(y) \in \sL_{J(y)}$.
Thus we obtain a holomorphic map $\lambda\colon  C^o \to \sL$ satisfying $J|_{C^o} = \varpi \circ \lambda$, which we call the {\em symplecto-functional invariant} of $(f\colon  X \to C, \Sigma, \omega)$. \end{definition}

\begin{definition}\label{d.mu}
In Definition \ref{d.sL}, suppose that we are given a holomorphic map  $\mu\colon R \to \sL$  from a Riemann surface $R$  satisfying the condition that the composition $\nu\coloneqq \varpi \circ \mu\colon R \to \sM^{\star}$ is unramified, hence $\mu$ itself is unramified.
For each $y \in R$,  the image $\mu(y) \in \sL_{\nu(y)}$  determines a lattice $\Lambda_y \subset T^*_{\nu(y)} \sM^{\star}$.
The family of elliptic curves  over $R$ $$\bigcup_{y \in R} (T^*_{\nu(y)} \sM^{\star})/\Lambda_y$$ defines a smooth elliptic fibration $(f^{\mu}\colon X^{\mu} \to R, \Sigma^{\mu})$ where  the section $\Sigma^{\mu} \subset X^{\mu}$ is given by the zero section of $T^* \sM^{\star}$.
Moreover, the standard symplectic form $\omega^{\sM^{\star}}$ on $T^* \sM^{\star}$ lifts
to a symplectic form $\omega^{\mu}$ on $X^{\mu}$ by Lemma \ref{l.standard}. We call $(f^{\mu}\colon X^{\mu} \to R, \Sigma^{\mu}, \omega^{\mu})$ the {\em symplectic elliptic fibration  determined by} $\mu$. \end{definition}

\begin{lemma}\label{l.mu}
Let  $(f\colon  X \to C, \Sigma, \omega)$ be a nonisotrivial symplectic elliptic fibration with the symplecto-functional invariant $\lambda\colon  C^o \to \sL$ as in Definition \ref{d.invariant}.
Let $(f^{\lambda}\colon X^{\lambda} \to C^o, \Sigma^{\lambda}, \omega^{\lambda})$ be the symplectic elliptic fibration  determined by $\lambda$ as in Definition \ref{d.mu}.
Then $(f^{\lambda}\colon X^{\lambda} \to C^o, \Sigma^{\lambda}, \omega^{\lambda})$ is isomorphic  to the restriction $(f\colon  X \to C, \Sigma, \omega)|_{C^o}$. \end{lemma}

\begin{proof}
Recall that $X^{\lambda}$ is constructed in Definition \ref{d.mu} as
 the quotient  of the total space of the line bundle $(\varpi \circ \lambda)^* (T^* \sM^{\star})$
on $C^o$  by a subbundle of lattices.
By the definition of the symplectic invariant $\lambda$, this quotient
 can be identified, via ${\rm d} J|_{C^o}$, with the quotient  $\chi\colon T^* C^o \to X^o = f^{-1}(C^o)$, which sends the zero section to $\Sigma \cap X^o$.
 Thus $(f^{\lambda}\colon X^{\lambda} \to C^o, \Sigma^{\lambda})$ and $(f^o\colon X^o \to C^o, \Sigma^o)$ are isomorphic as smooth elliptic fibrations.
 By Theorem \ref{t.AA}, this isomorphism sends $\omega^{\lambda}$ to $\omega|_{C^o}$.
 \end{proof}

We prove the following slight variation of Theorem \ref{t.lambda}.

\begin{theorem}\label{t.lambda*}
Let $(f\colon  X \to C, \Sigma, \omega)$ and $(\w{f}\colon  \w{X} \to \w{C}, \w{\Sigma}, \w{\omega})$ be two nonisotrivial symplectic elliptic fibrations.  Let $\lambda\colon  C^o \to \sL$ and $\w{\lambda}\colon  \w{C}^o \to \sL$ be their symplecto-functional invariants.  Suppose we have  Zariski-open subsets $C' \subset C^o$, $ \w{C}' \subset \w{C}^o$ and a biholomorphic map $\phi\colon C \to \w{C}$  such that  $$\phi(C') = \w{C}' \mbox{ and } \lambda|_{C'} = \w{\lambda} \circ \phi|_{C'}.$$ Then there exists a  biholomorphic map $\Phi\colon X \to \w{X}$ such that $(\phi, \Phi)$ defines an isomorphism from $(f\colon  X \to C, \Sigma, \omega)$ to $(\w{f}\colon  \w{X} \to \w{C}, \w{\Sigma}, \w{\omega})$. \end{theorem}

\begin{proof}
 Let us  identify $C$ with $\w{C}$, $C'$ with $\w{C}'$ and $\lambda'\coloneqq\lambda|_{C'}$ with $\w{\lambda}'\coloneqq \w{\lambda}|_{\w{C}'}$ by  the biholomorphic map $\phi\colon C \to \w{C}$. Then both of the restrictions \begin{eqnarray*} (f'\colon  X' \to C', \Sigma', \omega') &\coloneqq& (f\colon  X \to C, \Sigma, \omega)|_{C'} \mbox{ and } \\  (\w{f}'\colon \w{X}' \to
 C', \w{\Sigma}', \w{\omega}') & \coloneqq &  (\w{f}\colon  \w{X} \to
 C, \w{\Sigma}, \w{\omega})|_{C'} \end{eqnarray*}   are  isomorphic to $(f^{\lambda'}\colon X^{\lambda'} \to C', \Sigma^{\lambda'}, \omega^{\lambda'})$ by Lemma \ref{l.mu}.
Thus we have a biholomorphic map $\Phi'\colon X' \to \w{X}'$ that descends to $\phi'\coloneqq \phi|_{C'}$ such that $\Phi' (\Sigma') = \w{\Sigma}'$ and $(\Phi')^* \w{\omega}' = \omega'$. This implies that the functional invariants of $(f\colon  X \to C, \Sigma, \omega)$ and $(\w{f}\colon  \w{X} \to \w{C}, \w{\Sigma}, \w{\omega})$ are identical and so are their homological invariants because  both of the homomorphisms
$$\pi_1(C', c_0) \to \pi_1(C^{\rm reg}, c_0) \mbox{ and } \pi_1(\w{C}', c_0) \to \pi_1(\w{C}^{\rm reg}, c_0)$$
with a base point $c_0 \in C'$ are surjective. Since an elliptic fibration is determined by its functional and homological invariants (see \cite[Theorem 3.14, Chapter I]{FM}, noting that an elliptic fibration in our sense is an elliptic fibration with a section in \cite{FM}), we can extend $\Phi'$ to a biholomorphism $\Phi\colon X \to \w{X}$.
Note that \(\Phi\) sends  \(\Sigma\) to  \(\w{\Sigma}\). Since the holomorphic 2-form \(\Phi^*\w{\omega}-\omega\) vanishes over a  nonempty open subset in $X$, it vanishes identically on \(X\). Thus $(\phi, \Phi)$ defines an isomorphism of the two symplectic elliptic fibrations.
\end{proof}

We have the following result on the relation between isogenies and symplecto-functional invariants.

\begin{proposition}\label{p.invfactor}
Let $(\psi, \Psi)$ be an isogeny from $(f\colon  X\to C, \Sigma, \omega)$ to $(f'\colon  X' \to C', \Sigma', \omega')$. Then their symplecto-functional invariants $\lambda\colon  C^o \to \sL$ and $\lambda'\colon C'^o \to \sL$ satisfy $$\psi(C^o) \subset C'^o \mbox{ and } \lambda = \lambda' \circ \psi|_{C^o}.$$ \end{proposition}

\begin{proof}
It is clear that the functional invariants $J\colon  C \to \overline{\sM}$ and $J'\colon C' \to \overline{\sM}$ are related by $J = J' \circ \psi$ over the Zariski-open subset \(U\) defined in Definition \ref{d.isogeny} (ii); hence \(J=J'\circ\psi\) on the whole \(C\).
Therefore, if $y \in C^o$, then $J'(\psi(y)) \in \sM^{\star}$ and $J'$ is unramified at $\psi(y)$.
Consequently, we have $\psi(y) \in C'^o$, proving  $\psi(C^o) \subset C'^o$.

To show  $\lambda = \lambda'\circ \psi|_{C^o}$, it is sufficient to prove \(\lambda(y)=\lambda'(\psi(y))\) for \(y\in C^o\cap U\).
We have   the unramified holomorphic maps \(\chi_y\colon T_y^*C\to X_y\) and  \(\chi_{\psi(y)}\colon T_{\psi(y)}^*C'\to X'_{\psi(y)}\) from Definition \ref{d.chi}.
The lattices \(\chi_y^{-1}(\Sigma\cap X_y)\) and \(\chi_{\psi(y)}^{-1}(\Sigma'\cap X'_{\psi(y)})\) are sent to the same element in \(\BL(T^*_{J(y)} \sM^{\star})=\BL(T^*_{J'(\psi(y))} \sM^{\star})\) via the linear isomorphisms \(T_y^*C\cong T_{J(y)}^*\sM^{\star}\cong T^*_{\psi(y)}C'\)  induced by ${\rm d}_y J$ and ${\rm d}_{\psi(y)} J'$, which is compatible with the isomorphism induced by \({\rm d}_y\psi\).
This implies that \(\lambda(y)=\lambda'(\psi(y))\) for each \(y\in U\cap C^o\).
\end{proof}

\section{Symplecto-elliptic reduction of a nonisotrivial symplectic elliptic fibration}\label{s.reduction}

The main result of this section is the construction of the symplecto-elliptic reduction of a nonisotrivial symplectic elliptic fibration $(f\colon  X \to C, \Sigma, \omega)$. The key point is an extension of the symplecto-functional invariant $\lambda\colon  C^o \to \sL$ to a holomorphic map $\psi\colon  C \to \h{C}$ to a Riemann surface $\h{C}$. For this extension, we need the following definition.

\begin{definition}\label{d.hB}
Let $(f\colon  X \to C, \Sigma, \omega)$ be a nonisotrivial symplectic elliptic fibration  and let
$\lambda\colon  C^o \to \sL$ be its
symplecto-functional invariant (see Definition \ref{d.invariant}).
We say that two points $y_1, y_2 \in C$ are $\lambda$-{\em equivalent} if for any  neighborhood $y_1 \in U_1 \subset C$ (resp. $y_2 \in U_2 \subset C$), there exists a  neighborhood $y_1 \in O_1 \subset U_1$ (resp. $y_2 \in O_2 \subset U_2$) such that $\lambda (O_1 \cap C^o) = \lambda (O_2 \cap C^o) \subset \sL$. Write  \begin{itemize} \item[(i)]  $\h{C}$ for the set of $\lambda$-equivalence classes of points in $C$; \item[(ii)]   $\lambda[y] \in \h{C}$ for the $\lambda$-equivalence class of a point $y \in C$; and \item[(iii)] $\h{C}^{\flat} \coloneqq  \{ \lambda[y] \in \h{C} \mid y\in C^o\} \subset \h{C}$. \end{itemize} \end{definition}

It is convenient to use the following terminology.

\begin{definition}\label{d.generic} A holomorphic map $h\colon Y \to Z$ between complex manifolds is {\em generically injective}
if for a nonempty Zariski-open subset $Y^{*} \subset Y$, the restriction $h|_{Y^{*}}$ is injective.
\end{definition}

\begin{proposition}\label{p.hB}
In Definition \ref{d.hB},
the set $\h{C}$ has a  natural Riemann surface structure such that  $\h{C}^{\flat} \subset \h{C}$ is a Zariski-open subset and there are factorizations of the symplecto-functional invariant $\lambda$ and the functional invariant $J: C \to J(C) \subset \overline{\sM}$ into products of  surjective holomorphic maps $$
C^o \ \stackrel{\xi^{\flat}}{\longrightarrow} \ \h{C}^{\flat}  \ \stackrel{\lambda^{\flat}}{\longrightarrow} \ \lambda(C^o)  \  \mbox{ and } \   C \ \stackrel{\xi}{\longrightarrow} \ \h{C} \ \stackrel{\h{j}}{\longrightarrow} \ J(C) $$
such that  $$\lambda = \lambda^{\flat} \circ \xi^{\flat}, \ J = \h{j} \circ \xi, \   \xi^{\flat} = \xi|_{C^o}$$ and $\lambda^{\flat}$ is generically injective.
\end{proposition}

\begin{proof}
Let $\xi\colon  C \to \h{C}$ be the surjective map sending $y \in C$ to $\lambda[y] \in \h{C}$ and set $\xi^{\flat} = \xi|_{C^o}. $
 The set $\h{C}$ has a natural topology which makes $\xi$ a continuous map, and $\h{C}^{\flat} \subset \h{C}$ is an open subset in this topology.
Two points $y_1, y_2 \in C$ satisfy $\lambda[y_1] = \lambda[y_2]$ only if $J(y_1) = J(y_2)$.
 Thus $J$ induces a surjective continuous map $\h{j}\colon \h{C} \to J(C)$ with $J = \h{j} \circ \xi$.
 Two points $y_1, y_2 \in C^o$ satisfy $\lambda[y_1] = \lambda[y_2]$  only if $\lambda(y_1) = \lambda(y_2)$.
 Thus $\lambda$ induces a surjective continuous map $\lambda^{\flat}\colon \h{C}^{\flat} \to \lambda(C^o)$ with $\lambda = \lambda^{\flat} \circ \xi^{\flat}$.

If a point $x \in \lambda(C^o)$ admits an open neighborhood $x \in U \subset\lambda(C^o)$, then $U$  is a locally closed submanifold of $\sL$ and all points of $\lambda^{-1}(x)$ have the same $\lambda$-equivalence class. Thus
$\lambda^{\flat}$ is generically injective.

Each point $ y  \in C$ admits a neighborhood $y \in O_y \subset C$ such that $J|_{O_y}\colon O_y \to J(O_y)$ is biholomorphic to a $k_y$-cyclic covering map of the unit discs $(z \in \Delta_1) \mapsto (z^{k_y} \in \Delta_2)$ for some positive integer $k_y \geq 1$.
Then the factorization $$O^o_y\coloneqq O_y \cap C^o \stackrel{\lambda|_{O_y^o}}{\longrightarrow} \lambda(O_y^o) \stackrel{\varpi|_{\lambda(O_y^o)}}{\longrightarrow} J(O_y^o)$$ is a factorization of the $k_y$-cyclic covering $O_y^o \to J(O_y^o)$ into two unramified coverings.
Thus $\lambda|_{O_y^o}$ should be biholomorphic to the restriction of an $m_y$-cyclic covering $(z \in \Delta_1) \mapsto (z^{m_y} \in \Delta_2)$ of the unit discs  for some $1 \leq m_y \leq k_y$. Then an open neighborhood of $\h{C}$ at $\lambda[y]$ can be identified with $\Delta_2,$ the $m_y$-cyclic quotient of $O_y,$ and inherits a complex structure from $O_y$.
For $y, y' \in C$ with $\lambda[y] = \lambda[y']$, it is clear that the complex structure inherited from the $m_y$-cyclic quotient of a neighborhood $O_y$ of $y$ and that inherited from the $m_{y'}$-cyclic quotient of a neighborhood of $y'$ give the same complex structure on $\xi(O_y) \cap \xi(O_{y'})$.
Thus we obtain a well-defined Riemann surface structure on $\h{C}$. It is clear that $\h{C}^{\flat} \subset \h{C}$ is Zariski-open and the maps $\xi, \lambda^{\flat}$ and $\h{j}$ are holomorphic with respect to the complex structures on $\h{C}$ and $\h{C}^{\flat}.$
 \end{proof}

To construct the symplecto-elliptic reduction, it is convenient to use the following notion from \cite[Definition 1.1]{Be} and \cite[Definition 1.3]{Fu}.

\begin{definition}\label{d.singular}
Let $Y$ be a normal variety. A  symplectic form  $\omega_{\rm sm}$ on the smooth locus $Y_{\rm sm}$ of $Y$ is  called a {\em symplectic form on } $Y$   if  for any desingularization \(\pi\colon W \to Y\) of $Y$,  the pullback of \(\omega_{\textup{sm}}\) to \(\pi^{-1}(Y_{\textup{sm}})\) extends to a holomorphic 2-form on \(W\). A normal variety equipped with a symplectic form is called a {\em symplectic variety}.
A desingularization $\pi\colon W \to Y$ of a symplectic variety $(Y, \omega_{\rm sm})$ is said to be a {\em symplectic resolution} if the extended holomorphic 2-form of the pullback of $\omega_{\rm sm}$ on $W$ is a symplectic form.
\end{definition}

The following proposition is borrowed from \cite[Propositions 1.3 and 2.4]{Be} and \cite[Proposition 1.1]{Fu}.

\begin{proposition}\label{p.Beauville}
\begin{enumerate}
\item[(i)] A symplectic variety is rational Gorenstein. In particular, the minimal resolution of a two-dimensional symplectic variety is a crepant resolution, which is a symplectic resolution.
\item[(ii)]  Let $(Y, \omega_{\rm sm})$ be a symplectic variety. Suppose that   a finite group $G$ acts on $Y$  preserving  $\omega_{\rm sm}$ on $Y_{\textup{sm}}$.
Then $\omega_{\rm sm}$ descends to a symplectic form on the quotient $Y/G$.
\end{enumerate}
\end{proposition}

Recall the following elementary lemma.

\begin{lemma}\label{l.id}
Let $(f\colon X \to C, \Sigma, \omega)$ be a symplectic elliptic fibration.
Let $(\phi\colon C \to C, \Phi\colon X \to X)$ be an automorphism, namely, a pair of biholomorphic automorphisms such that $f \circ \Phi = \phi \circ f, \Phi(\Sigma) = \Sigma$ and $ \Phi^* \omega = \omega.$ If $\phi = {\rm Id}_C$, then $\Phi = {\rm Id}_X$. \end{lemma}

\begin{proof}
Let $X_t$ be a smooth fiber of $f$ and set $s := X_t \cap \Sigma$.  The induced map $\Phi_t \coloneqq \Phi|_{X_t}$ is an automorphism of the elliptic curve fixing \(s\). So ${\rm d}_s \Phi_t$ acts on $T_s X_t$ as a multiplication by a constant $c \in \mathbb{C}^*$. 
Then the linear automorphism ${\rm d}_{s} \Phi$ of $T_{s} X$ acts with the eigenvalue $1$ on $T_{s} \Sigma$ and $c$ on $T_{s} X_t$. Since it fixes $\omega_{s} \in \wedge^2 T_s X$, we see $c =1$. Thus $\Phi_t = {\rm Id}_{X_t}$ for any smooth fiber $X_t$.   It follows that $\Phi = {\rm Id}_X$. \end{proof}

\begin{proposition}\label{p.descent}
Let $d >1$ be a positive integer and let $\zeta \in \C$ be a primitive $d$-th root of unity.
Let $(f\colon   S \to \Delta, \Sigma, \omega)$  be a symplectic elliptic fibration over the unit disc $\Delta = \{ t \in \C \mid |t| <1\}$. Let $\sigma\colon  S \to S$ be a biholomorphic automorphism  satisfying $$\sigma(\Sigma) = \Sigma, \ \sigma^* \omega = \omega \mbox{ and } f(\sigma (x)) =  \zeta \cdot f(x) \mbox{ for all } x \in S.$$
Then \(\sigma\) is  of finite order \(d\) and  we have the following commutative diagram
\[
\xymatrix{
(S,\Sigma, \omega)\ar@/^2pc/@{-->}[rr]^-{\Psi}
\ar[r]^p \ar[d]_f&\widetilde{S}
\ar[d]_{\widetilde{f}}
&(S',\Sigma',\omega')\ar[l]_{p'}\ar[dl]^{f'}\\
\Delta\ar[r]^{\psi}&\Delta'&
}
\]
where \begin{itemize} \item $(f'\colon  S' \to \Delta', \Sigma', \omega')$ is a symplectic elliptic fibration over the unit disc $\Delta'$;   \item $p$ and $ \psi$ are the quotients by the cyclic group $\Z/d \Z$ generated by $\sigma$; \item $\w{f}$ is the induced map; and \item \( p'\) is the minimal resolution of the surface $\w{S}$, \end{itemize} such that $\psi$ and the induced dominant meromorphic map $\Psi\colon S \dasharrow S'$ define an isogeny from $(f\colon   S \to \Delta, \Sigma, \omega)$ to $(f'\colon  S' \to \Delta', \Sigma', \omega')$.
\end{proposition}

\begin{proof}
 Since $\sigma^d$ induces the identity of the base $\Delta$ and fixes the section $\Sigma$, Lemma \ref{l.id}  implies that \(\sigma^d\) is identity on $S$.

By Proposition \ref{p.Beauville} (ii), the symplectic form \(\omega\) on \(S\) descends to a symplectic form on \(\widetilde{S}\) and its pullback  by $p'$ can be extended to a holomorphic 2-form \(\omega'\) on \(S'\).
By Proposition \ref{p.Beauville} (i), the minimal resolution $p'$ is a symplectic resolution and $\omega'$ is a symplectic form on $S'$. The proper image $\Sigma'$ of $\Sigma$ under $\Psi$ is clearly a section of $f'$.
\end{proof}

 \begin{theorem}\label{t.hX}
 Let $(f\colon  X \to C, \Sigma, \omega)$ be a nonisotrivial symplectic elliptic fibration.
 Let $\xi\colon  C \to \h{C}$, $\lambda^{\flat}\colon \h{C}^{\flat} \to \lambda(C^o)$, and $\h{j}\colon \h{C} \to J(C)$ be as in Proposition \ref{p.hB}. Then there exists a unique (up to isomorphisms) nonisotrivial symplectic elliptic fibration
 $(\h{f}\colon  \h{X} \to \h{C}, \h{\Sigma}, \h{\omega})$ such that
 \begin{itemize} \item[(i)] its functional invariant $\h{J}\colon \h{C} \to \overline{\sM}$ coincides with $\h{j}$; \item[(ii)] the domain $\h{C}^o$ of its symplecto-functional invariant $\h{\lambda}\colon  \h{C}^o \to \sL$ contains $\h{C}^{\flat}$ as a Zariski-open subset; \item[(iii)] $\lambda^{\flat} = \h{\lambda}|_{\h{C}^{\flat}}$;
 \item[(iv)]
 $\h{\lambda}\colon  \h{C}^o \to \sL$ is generically injective; and
  \item[(v)]
  there is  an  isogeny $(\xi, \Xi)$ from $(f\colon  X \to C, \Sigma, \omega)$ to $(\h{f}\colon  \h{X} \to \h{C}, \h{\Sigma}, \h{\omega}).$
\end{itemize}
\end{theorem}

\begin{proof}
Since $\lambda^{\flat} \colon \h{C}^{\flat} \to \lambda(C^o) \subset \sL$ is unramified, Definition \ref{d.mu} gives the symplectic elliptic fibration $(f^{\lambda^{\flat}}\colon X^{\lambda^{\flat}} \to \h{C}^{\flat}, \Sigma^{\lambda^{\flat}}, \omega^{\lambda^{\flat}})$ determined by \(\lambda^{\flat}\).
By Lemma \ref{l.mu}, we have a canonical isogeny $(\xi^{\flat}, \Xi^{\flat})$ from $(f\colon  X \to C, \Sigma, \omega)|_{C^o}$ to $(f^{\lambda^{\flat}}\colon X^{\lambda^{\flat}} \to \h{C}^{\flat}, \Sigma^{\lambda^{\flat}}, \omega^{\lambda^{\flat}})$  via the \(\xi^{\flat} \)-pullback.

We would like to extend $(f^{\lambda^{\flat}}\colon  X^{\lambda^{\flat}} \to \h{C}^{\flat}, \Sigma^{\lambda^{\flat}}, \omega^{\lambda^{\flat}})$ to a symplectic elliptic fibration over $\h{C}$, together with an extension of the isogeny.

For each point $z \in \h{C} \setminus \h{C}^{\flat}$, choose a point $y \in \xi^{-1}(z)$. We can find a neighborhood $y \in U \subset C$ such that
 \begin{itemize}
 \item $U \setminus \{y\} \subset C^o$;
\item $U$ is biholomorphic to a unit disc $\Delta \subset \C$; and
\item $J\colon  U \to J(U)$ is equivalent to the cyclic covering  $(t \in \Delta) \mapsto (t^{k} \in \Delta)$ for some $k \geq 1$. \end{itemize}
Let us write
$$U^o\coloneqq U \cap C^o,  \ \h{U} \coloneqq \xi(U) \subset \h{C} \mbox{ and } \h{U}^{\flat} \coloneqq \h{U} \setminus \{z\}\subset \h{C}^{\flat}.$$
Then $\xi|_U \colon  U \to \xi(U)$ is equivalent to an $m$-cyclic covering for some $m \leq k$, namely,  there exists a biholomorphic automorphism  $v\colon U \to U$ of order $m$ such that $\xi|_U\colon  U \to \h{U}$ can be identified with the quotient of $U$ by the finite cyclic group generated by $v$. Then \(\lambda\) is \(v|_{U^o}\)-invariant, i.e., $\lambda \circ v|_{U^o} = \lambda|_{U^o}$.
By Theorem \ref{t.lambda}, the biholomorphic map $v$  can be lifted to an automorphism of $(f\colon  X \to C, \Sigma, \omega)|_U$.
Then we can apply Proposition \ref{p.descent}  to obtain  a symplectic elliptic fibration $(\h{f}_U\colon \h{X}_U \to \h{U}, \h{\Sigma}_U, \h{\omega}_U)$ equipped with an isogeny $(\xi_U, \Xi_U)$ from $(f\colon  X \to C, \Sigma, \omega)|_U.$
By Proposition \ref{p.invfactor}, its symplecto-functional invariant $\h{\lambda}_U\colon \h{U}^o \to \sL$ satisfies $$\h{U}^{\flat} \subset \h{U}^o, \ \lambda|_{U^o} = \h{\lambda}_U \circ \xi|_{U^o} \mbox{ and } \h{\lambda}_U|_{\h{U}^{\flat}} = \lambda^{\flat}|_{\h{U}^{\flat}},$$
where the last equality follows from the second-last equality and  \(\lambda=\lambda^{\flat}\circ\xi^{\flat}\) in Proposition \ref{p.hB}.
  The last equality and Theorem \ref{t.lambda*} imply that $(\h{f}_U\colon \h{X}_U \to \h{U}, \h{\Sigma}_U, \h{\omega}_U)|_{\h{U}^{\flat}}$ and $(f^{\lambda^{\flat}}\colon X^{\lambda^{\flat}} \to \h{C}^{\flat}, \Sigma^{\lambda^{\flat}}, \omega^{\lambda^{\flat}})|_{\h{U}^{\flat}}$ are isomorphic.
    Thus we can glue them together to extend \(f^{\lambda^{\flat}}\) to  a symplectic elliptic fibration over $\h{C}^{\flat} \cup \{z\}.$
  By Theorem \ref{t.lambda*}, this construction does not depend on the choice of $y \in \xi^{-1}(z)$ and $U$:   the construction from a different choice is related to the above one by a unique biholomorphic map. Repeatedly applying this construction for each $z \in  \h{C} \setminus \h{C}^{\flat}$, we obtain the symplectic elliptic fibration $(\h{f}\colon  \h{X} \to \h{C}, \h{\Sigma}, \h{\omega})$.

    The properties (i), (ii), and (iii) are from the construction. (iv) follows from  the fact that $\lambda^{\flat}$ is generically injective in Proposition \ref{p.hB}.
    To check (v), note that for each choice of \(z\in \h{C} \setminus \h{C}^{\flat}\) and   $y \in U$, the isogeny $(\xi_U, \Xi_U)|_{U^o}$ coincides with $(\xi^{\flat}, \Xi^{\flat})|_{U^o}$.
    So we can patch them to obtain the desired isogeny $(\xi, \Xi).$

    Finally, the uniqueness of $(\h{f}\colon  \h{X} \to \h{C}, \h{\Sigma}, \h{\omega})$ follows from (iii) and Theorem \ref{t.lambda*}. \end{proof}

Let us call $(\h{f}\colon  \h{X} \to \h{C}, \h{\Sigma}, \h{\omega})$, the symplecto-elliptic reduction of $(f\colon  X \to C, \Sigma, \omega)$ (see Definition \ref{d.reduction}).
The following proposition proves Theorem \ref{t.factorize}.

\begin{proposition}\label{p.reduction}
Let  $(f\colon  X \to C, \Sigma, \omega)$ (resp. $(f'\colon  X' \to C', \Sigma', \omega')$) be a nonisotrivial symplectic elliptic fibration and let $(\h{f}\colon  \h{X} \to \h{C}, \h{\Sigma}, \h{\omega})$ (resp. $(\h{f}'\colon \h{X}' \to \h{C}', \h{\Sigma}', \h{\omega}')$) be its symplecto-elliptic reduction, equipped with the isogeny $(\xi\colon  C \to \h{C}, \Xi\colon  X \dasharrow \h{X})$ (resp. $(\xi'\colon  C' \to \h{C}', \Xi'\colon  X' \dasharrow \h{X}')$).
If there is an isogeny $(\psi, \Psi)$ from $(f\colon  X \to C, \Sigma, \omega)$ to  $(f'\colon  X' \to C', \Sigma', \omega')$, then there is an isomorphism $(\phi\colon \h{C} \to  \h{C}', \Phi\colon \h{X} \to \h{X}')$ such that  $\phi\circ \xi = \xi' \circ \psi$ and $\Phi\circ \Xi = \Xi' \circ \Psi.$
\end{proposition}

\begin{proof}
 The surjective holomorphic map $\psi\colon C \to C'$ sends an open subset in $C$ to an open subset in $C'$. Thus by Proposition \ref{p.invfactor}, we see that the $\psi$ sends a $\lambda$-equivalence class of points on $C$ to a $\lambda'$-equivalence class of points on $C'$ in the sense of Definition \ref{d.hB}, inducing a map $\phi\colon \h{C} \to \h{C}'$.
Since \(\psi\) is continuous, the $\psi$-inverse image of  a \(\lambda'\)-equivalence class of points on \(C'\) is  a \(\lambda\)-equivalence class of points on \(C\). Thus $\phi$ is bijective. This is clearly a biholomorphic map with the complex structures on $\h{C}$ and $\h{C}'$ in Proposition \ref{p.hB} and satisfies $\phi \circ \xi = \xi' \circ \psi$.
Let $\h{\lambda}\colon  \h{C}^o \to \sL$ and $\h{\lambda}'\colon \h{C}' \to \sL$ be the symplecto-functional invariants of the symplecto-elliptic reductions.
From the condition (c) of Definition \ref{d.isogeny}, we see that   $\h{\lambda}$ and  $\h{\lambda}' \circ \phi$ agree on a Zariski-open subset \(\xi(C^o)\) of $\h{C}^o$. Thus Theorem \ref{t.lambda*} gives an isomorphism $(\phi, \Phi)$ satisfying $\xi = \xi' \circ \psi$ and $\Xi = \Xi' \circ \Psi.$
\end{proof}

Let us close this section with the proof of Theorem \ref{t.K3}, an application to elliptic  K3 surfaces. The following lemma is straightforward.

\begin{lemma}\label{l.K3} Let $(f\colon  X \to C, \Sigma, \omega)$ be a nonisotrivial symplectic elliptic fibration with the symplecto-elliptic reduction $(\h{f}\colon  \h{X} \to \h{C}, \h{\Sigma}, \h{\omega}).$  Assume that $X$ is compact. Then both $C$ and $\h{C}$ are biholomorphic to the Riemann sphere $\BP^1$ and both $X$ and $\h{X}$ are K3 surfaces.
\end{lemma}

\begin{proof}[Proof of Theorem \ref{t.K3}]
Let $(\h{f}_i\colon  \h{X}_i \to \h{C}_i, \h{\Sigma}_i, \h{\omega}_i)$ be the symplecto-elliptic reduction of $(f_i\colon  X_i \to C_i, \Sigma_i, \omega_i)$ with \(i=1,2\).
By Lemma \ref{l.K3}, we have $C_1 \cong \BP^1 \cong C_2$. Then the functional invariants $J_i\colon  C_i \to \overline{\sM}$ are finite surjective holomorphic maps.
Let \(\lambda_i\) (resp.~\(\h{\lambda}_i\)) be the symplecto-functional invariants of \(f_i\) (resp.~\(\h{f}_i\)) with \(i=1,2\).
Then we can find a Zariski-open subset $\sM' \subset \sM^{\star}$ such that
\begin{itemize}
\item $C_i^*\coloneqq J_i^{-1}(\sM') \subset C_i^o$ for  \(i=1,2\); 
\item $\lambda_1(C_1^*)$ (resp.~$\lambda_2(C_2^*)$) is a  connected closed analytic subset  in  $\sL|_{\sM'} =\varpi^{-1}(\sM')$.
\end{itemize}
By the isomorphism $(\phi, \Phi)$ over $O_1 \subset C_1$ and $O_2 \subset C_2$, the intersection of the two connected closed analytic subsets $\lambda_1(C_1^*)$ and $\lambda_2(C_2^*)$ in $\varpi^{-1}(\sM')$ contains $\lambda_1 (O_1 \cap C_1^*) \subset \lambda_1(C_1^*)$ (resp. $\lambda_2(O_2 \cap C_2^*) \subset \lambda_2(C_2^*)$) which has nonempty interior in $\lambda_1(C_1^*)$ (resp. $\lambda_2(C_2^*)$).
It follows that $\lambda_1(C_1^*)= \lambda_2(C_2^*).$
But $\lambda_1(C_1^*)$ (resp. $\lambda_2(C_2^*)$) is  Zariski-open in $\h{\lambda}_1(\h{C}_1^o)$
(resp. $\h{\lambda}_2(\h{C}_2^{o})$). Since both $\h{\lambda}_1$ and $\h{\lambda}_2$ are generically injective and their images contain the same Zariski-open subset, they induce a biholomorphic map $\varphi\colon  \h{C}_1 \to \h{C}_2$  satisfying $\h{\lambda}_1 = \h{\lambda}_2 \circ \varphi$ on  a Zariski-open subset of $\h{C}_1$.
By Theorem \ref{t.lambda*}, we obtain an isomorphism between $(\h{f}_1\colon  \h{X}_1 \to \h{C}_1, \h{\Sigma}_1, \h{\omega}_1)$ and $(\h{f}_2\colon \h{X}_2 \to \h{C}_2, \h{\Sigma}_2, \h{\omega}_2)$, proving (i).  (ii) follows from the construction of the isomorphism.
(iii) is a consequence of (i) and (ii). \end{proof}

\section{Automorphisms of finite order of germs of symplectic elliptic fibrations}\label{s.auto}

To understand the local structure of isogenies of symplectic elliptic fibrations, we need to understand symplectic automorphisms of finite order of germs of symplectic elliptic fibrations. The goal of this section is to classify germs of elliptic fibrations admitting symplectic structures with automorphisms of finite order.

\begin{definition}\label{d.germ}
Let $(f\colon  X \to C, \Sigma)$ (resp. $(\w{f}\colon \w{X} \to \w{C}, \w{\Sigma}))$ be a relatively minimal elliptic fibration and let $J\colon  C \to \overline{\sM}$ (resp. $\w{J}\colon  \w{C} \to \overline{\sM}$) be its functional invariant. We say that the {\em germ} of  $(f\colon  X \to C, \Sigma)$ at a point $y \in C$ is {\em equivalent to the germ} of $(\w{f}\colon \w{X} \to \w{C}, \w{\Sigma})$ at a point $\w{y} \in \w{C}$, if  there are neighborhoods $y \in O \subset C$ and $\w{y} \in \w{O} \subset \w{C}$ such that $(f\colon  X \to C, \Sigma)|_O$ is isomorphic to $(\w{f}\colon  \w{X} \to \w{C}, \w{\Sigma})|_{\w{O}}$. \end{definition}

Recall the following well-known fact (for example, see \cite[Proposition VI.1.1]{Mi}).

\begin{lemma}\label{l.germ}
In Definition \ref{d.germ},  the germ of  $(f\colon  X \to C, \Sigma)$ at a point $y \in C$ is equivalent to the germ of $(\w{f}\colon \w{X} \to \w{C}, \w{\Sigma})$ at a point $\w{y} \in \w{C}$, if and only if \begin{itemize} \item the Kodaira's singular fiber types of $f$ at $y$ and  $\w{f}$ at $\w{y}$ agree;
\item  $J(y)=\w{J}(\w{y})$, namely, the values of the functional invariants agree;
and \item ${\rm mult}_y J = {\rm mult}_{\w{y}} \w{J},$ namely, the multiplicities of the functional invariants agree.
\end{itemize}  In particular, we can represent the germ-equivalence class of $(f\colon  X \to C, \Sigma)$ at $y \in C$ by the triple $(*,  J(y), {\rm mult}_y J)$ consisting of the Kodaira singular fiber type $*$ of $f^{-1}(y)$, the value of   $J(y) \in \C \cup \{ \infty \}$ and the integer  ${\rm mult}_y J$. \end{lemma}

We have the following classification result.

\begin{theorem}\label{t.auto}
Let $(f\colon  S \to \Delta, \Sigma, \omega)$ be a symplectic elliptic fibration over the unit disc.  Fix a positive integer $d>1$ and let $\zeta \in \C$ be a primitive $d$-th root of unity.
Assume that there exists a biholomorphic automorphism $\sigma\colon  S\to S$ of finite order $d>1$ such that
$$ f \circ \sigma (y)  =  \zeta \cdot f(y), \ \sigma(\Sigma) = \Sigma  \ \mbox{ and } \sigma^* \omega = \omega.$$ Then the   triple $(*,   J(0), {\rm mult}_0 J)$ for the germ of $(f\colon  S \to \Delta, \Sigma)$ at $0$ and the value of $d$ must be one of the following. Here,  $k$ stands for a positive integer or $\infty$, except $k \neq \infty$ in (7), and  ${\rm mult}_0 J = \infty$ means that $f$ is isotrivial.
\begin{itemize}
\item[(1)] $(\textup{I}_{0}, J(0), 2k), d =2,$ where $J(0)\neq 0,1,\infty$;
\item[(2)] $(\textup{I}_0, 0,  6k), d= 2$;
\item[(3)] $(\textup{I}_0, 0, 9k-6), d= 3 $;
\item[(4)] $(\textup{I}_0, 0, 18k-6), d=6$;
\item[(5)] $(\textup{I}_0, 1,4k), d=2$;
\item[(6)] $(\textup{I}_0, 1, 8k-4), d= 4$;
\item[(7)]  $(\textup{I}_{2k}, \infty,  2k), d= 2$;
\item[(8)] $(\textup{I}_0^*, 0, 9k-3), d=3$;
\item[(9)] $(\textup{I}_0^*, 1, 4k-2), d=2$;
\item[(10)] $(\textup{II}, 0, 15k-5), d=5$;
\item[(11)] $(\textup{III}, 1, 6k-3), d= 3$;
\item[(12)] $(\textup{IV}, 0, 6k-4), d=2$;
\item[(13)] $(\textup{IV}, 0, 12k-4), d=4$;
\item[(14)] $(\textup{IV}^*, 0, 6k-2), d=2$.
\end{itemize}
\end{theorem}

For the proof, we need to recall the notion of the Weierstrass model of an elliptic fibration (see \cite[Section 1.4.1]{FM} or \cite[Lecture II]{Mi}).

\begin{definition}\label{d.Weierstrass}
Let $(f\colon  X \to C, \Sigma)$ be a relatively minimal elliptic fibration.  All components of singular fibers of $f$ disjoint from $\Sigma$ can be contracted to give a fibration $\overline{f}\colon  \overline{X} \to C$ of a singular surface $\overline{X}$ with a section $\overline{\Sigma} \subset \overline{X}$ such that all fibers of $\overline{f}$ are irreducible curves of arithmetic genus 1. This fibration is called the {\em Weierstrass model} of $(f\colon  X \to C, \Sigma)$. \end{definition}

    The following lemma is well-known (see \cite[Theorem 4.1 in Section 1.4.1]{FM} or \cite[Lemma II.5.3]{Mi}).

    \begin{lemma}\label{l.Weierstrass}
    In Definition \ref{d.Weierstrass}, let $L$ be the line bundle on $C$ induced by the conormal bundle $N^*_{\Sigma \subset X}$ on $\Sigma$  via $f|_{\Sigma}\colon \Sigma \cong C$ and let $\W\coloneqq L^{\otimes 2} \oplus L^{\otimes 3} \oplus \sO_C$, a vector bundle on $C$.   \begin{itemize}
    \item[(i)]
There is a pair $(\alpha, \beta)$ of holomorphic sections $\alpha \in H^0(C, L^{\otimes 4})$ and $\beta \in H^0(C, L^{\otimes 6})$ such that
the Weierstrass model $(\overline{f}\colon  \overline{X} \to C, \overline{\Sigma})$ is isomorphic, as a fibration of a surface with a section, to the surface  in $\BP \W$
defined by the equation $$ y^2 z = x^3 + \alpha xz^2 + \beta z^3,$$ with the fibration given by the natural projection to $C$ and the section defined by $x=z=0$.
\item[(ii)] Any two pairs $(\alpha,\beta)$ and $(\alpha', \beta')$ of sections of $L^{\otimes 4}$ and $L^{\otimes 6}$ with the property in (i) are related by $$(\alpha', \beta') = (g^4 \cdot \alpha, g^6 \cdot \beta)$$ for some nowhere vanishing holomorphic function $g$ on $C$.\end{itemize}
    \end{lemma}

The following is well-known (see \cite[Sections 5.10 and  5.13]{SS}).

\begin{lemma}\label{l.KS}
Let $(f\colon  S \to \Delta, \Sigma)$ be a relatively minimal elliptic fibration over the unit disc $\Delta = \{ t \in \C \mid |t| <1\}$.
We can trivialize the line bundle $L$ from Lemma \ref{l.Weierstrass} and  regard $\BP\W$ as $\BP^2 \times \Delta$. By this trivialization, we can write $(\alpha, \beta)$ in Lemma \ref{l.Weierstrass} (i) as a pair of holomorphic functions $(a(t), b(t))$ on $\Delta$ so that the Weierstrass model is embedded in $\BP^2 \times \Delta$ as the surface $$y^2 =  x^3 + a(t) x + b(t)$$ in a relative affine coordinate $(x,y)$ on $\BP^2$. Then the following holds.
\begin{itemize}
 \item[(i)]  The functional invariant $J\colon  \Delta \to \BP^1$ of the elliptic fibration $f\colon  S \to \Delta$ satisfies $$J(t) = \frac{4 a(t)^3}{4 a(t)^3 + 27 b(t)^2}.$$
      \item[(ii)] For each point $y \in \Delta$, either  ${\rm ord}_y a(t) <4$ or  ${\rm ord}_y b(t) <6$.
      \item[(iii)] The meromorphic 1-form $\frac{{\rm d} x \wedge {\rm d} t}{y}$ gives a symplectic form on the smooth locus of the surface $y^2 = x^3 + a(t) x + b(t)$ in  $\BP^2 \times \Delta$, which can be lifted to a symplectic form on $S$.
      \item[(iv)] The isomorphism type of the germ of $f$ at $y$ (equivalently, the triple $(*, J(y), {\rm mult}_y J)$) gives further restrictions on ${\rm ord}_y a$ and ${\rm ord}_y b$, which are listed in \cite[Table VI.1.2]{Mi} or \cite[Table 5.1 in Section 5.9]{SS}. \end{itemize}
      \end{lemma}

The following simple lemma is useful in the proof of Theorem \ref{t.auto}.

\begin{lemma}\label{l.lift}
In Lemma \ref{l.Weierstrass}, let $\sigma$ be an automorphism of an elliptic fibration $(f\colon  X \to C, \Sigma)$. Then we can lift it to an automorphism of the projective bundle $\BP \W \to C,$   which preserves the surface $y^2z = x^3+ \alpha xz^2 + \beta z^3$ such that the induced automorphism of the Weierstrass model $(\overline{f}\colon  \overline{X} \to C, \overline{\Sigma})$
agrees with the one coming from the contraction $X \to \overline{X}.$\end{lemma}

\begin{proof}
As $\sigma$ preserves $\Sigma \subset X$, it acts naturally on the line bundle $L$. Thus it induces an automorphism   $\sigma'\colon \BP\W  \to \BP\W,$ which  sends the surface $\overline{X} = (y^2z = x^3+ \alpha xz^2 + \beta z^3)$  to another surface $\overline{X}'= (y^2z = x^3+ \alpha' xz^2 + \beta' z^3)$ by Lemma \ref{l.Weierstrass} (ii).   Then we can find a projective bundle automorphism $\varphi$ of  $\BP \W$ which induces the identity on the base $C$ such that $\varphi(\overline{X}') = \overline{X}$. Then $\varphi \circ \sigma'$ is an automorphism of $\BP \W$ preserving the surface $\overline{X}$ and its restriction on $\overline{X}$ is the automorphism induced by $\sigma$ via the contraction $X \to \overline{X}$. \end{proof}

\begin{proof}[Proof of Theorem \ref{t.auto}]
Let us use the notation of Lemma \ref{l.KS} and use Lemma \ref{l.lift} to consider the automorphism induced by $\sigma$ on the Weierstrass model of $f\colon  S \to \Delta$. We use the same symbol $\sigma$ to denote the automorphism on $\BP^2 \times \Delta$ preserving the surface $y^2 = x^3 + a(t) x + b(t)$.  Then \(\sigma\) sends \((t,x,y)\) to \((\zeta \cdot t,u(t)^2x,u(t)^3y)\) for some \(\zeta\in\mathbb{C}^*\), and some holomorphic function \(u(t)\) (see e.g. \cite[Page 42, (3.8)]{SS}). Therefore, for any \((t,x,y)\in \overline{S}\) we have
\[
u(t)^6y^2=u(t)^6x^3+a(\zeta \cdot t)u(t)^2x+b(\zeta \cdot t).
\]
This forces that
\begin{equation}\label{eq-isom-type}
a(\zeta\cdot t)=u(t)^4a(t),~b(\zeta\cdot t)=u(t)^6b(t).
\end{equation}
Let $\overline{\omega}$ be the symplectic form on the smooth locus of $\overline{X}$ induced by $\omega$ so that $\sigma^* \overline{\omega} = \overline{\omega}.$
By Lemma \ref{l.KS} (iii),  we can write
\[\overline{\omega}=h(t)\cdot \frac{{\rm d}x\wedge {\rm d}t}{y}\]
for some nowhere vanishing holomorphic function \(h(t)\).
Then
$$
h(\zeta\cdot t)\cdot \frac{\zeta\cdot u(t)^2}{u(t)^3}\frac{{\rm d}x\wedge {\rm d}t}{y}=\sigma^*\overline{\omega}=\overline{\omega}=h(t)\cdot \frac{{\rm d}x\wedge {\rm d}t}{y}.
$$
 This implies that
\begin{equation}\label{eq-symplec-condition-equi}
\zeta\cdot h(\zeta \cdot t)=u(t)\cdot h(t).
\end{equation}
Combining \cref{eq-isom-type} and \cref{eq-symplec-condition-equi}, we have \begin{equation}\label{eq.ab}
a(\zeta \cdot t) = \zeta^{4} \frac{ h(\zeta \cdot t)^4}{h(t)^4} a(t) \mbox{ and } b(\zeta \cdot t) = \zeta^{6} \frac{h(\zeta  \cdot t)^6}{h(t)^6} b(t). \end{equation} In particular, \begin{equation}\label{eq.limit} \mbox{ either } a(t) \equiv 0 \mbox{ or } \zeta^4 = \lim_{t\to 0} \frac{a(\zeta \cdot t)}{a(t)} \end{equation}  and \begin{equation}\label{eq.limitb}  \mbox{ either } b(t) \equiv 0 \mbox{ or } \zeta^6 = \lim_{t\to 0} \frac{b(\zeta \cdot t)}{b(t)}. \end{equation}
Also, we have the obvious relation \begin{equation}\label{eq.J} J( \zeta \cdot t) = J(t). \end{equation}

We go through the list in \cite[Table VI.1.2]{Mi} or \cite[Table 5.1 in Section 5.9]{SS} mentioned in Lemma \ref{l.KS} (iv) and check which cases are allowed under \cref{eq.ab} and \cref{eq.J}.

\textbf{The case when \(S_0\) is of type \(\textup{I}_0\).}
Then $J(0) \neq \infty$ and one of the following cases occurs.
\begin{enumerate}
\item[(i)] \(a(0)\neq 0\) and \(b(0)\neq 0\) (thus, \(J(0)\neq 0,1\));
\item[(ii)] \(a(0)=0\) and \(b(0)\neq 0\) (thus, \(J(0)=0\)); or
\item[(iii)] \(a(0)\neq 0\) and \(b(0)=0\) (thus, \(J(0)=1\)).
\end{enumerate}

For the case (i), \cref{eq.limit} and \cref{eq.limitb} give \(\zeta^4=\zeta^6=1\) which implies \(\zeta=-1\) and $d=2$.
Thus by \cref{eq.J}, the order of vanishing of $J(t) - J(0)$ must be an even number. This case belongs to (1).

For the case (ii),  \cref{eq.limitb} gives \(\zeta^6=1\) and $d = 2, 3$ or $6$.
For each possible value of $d$,  \cref{eq.ab} implies either \(a(t)\equiv 0\)  or  $$ {\rm mult}_0 a(t) \equiv 4 \mod d.$$   Thus $ {\rm mult}_0 a(t)$ is of the form \(2k\), \(3k-2\), or \(6k-2\) for some \(k\geq 1\),  depending on $d = 2, 3$ or $6$, respectively.
By  Lemma \ref{l.KS} (i), either  \(J(t)\equiv 0\) or ${\rm mult}_0 J(t)$ is of the form \(6k\), \(9k-6\), or \(18k-6\) with \(k\geq 1\), depending on $d = 2, 3$ or $6$, respectively.
This case is (2), (3), or (4).

For the case (iii), \cref{eq.limit} gives \(\zeta^4=1\) and $d = 2$ or $4$.
For each possible value of $d$, \cref{eq.ab} implies either \(b(t)\equiv 0\)  or  $$ {\rm mult}_0 b(t) \equiv  6 \mod d.$$   Thus $ {\rm mult}_0 b(t)$ is of the form \(2k\), or \(4k-2\) for some \(k\geq 1\),  depending on $d = 2$ or $4$, respectively.
By  Lemma \ref{l.KS} (iii), either  \(J(t)\equiv 1 \) or  ${\rm mult}_0 J(t)$ is of the form \(4k\) or \(8k-4\),  depending on $d = 2$ or $4$, respectively. This case is (5) or (6).

\textbf{The case when \(S_0\) is of type \(\textup{I}_N\) with \(N\geq 1\).}
Then $J(0) = \infty$, \(a(0)\neq 0\) and \(b(0)\neq 0\).
\cref{eq.limit} and \cref{eq.limitb} give \(\zeta^4= \zeta^6=1\) which implies \(d=2\).
Thus ${\rm mult}_0 J$ must be an even number by  \cref{eq.J}. This case is  (7).

\textbf{The case when \(S_0\) is of type \(\textup{I}_0^*\).} Then $J(0) \neq \infty$, $a(0) = b(0) =0$ and one of the following cases occurs.
\begin{enumerate}
\item[(i)] $J(0) \neq 0,1$, ${\rm mult}_0 a(t) = 2$ and ${\rm mult}_0 b(t) =3$;
\item[(ii)]  $J(0) =0$, ${\rm mult}_0 a(t) \geq 3 $ and ${\rm mult}_0 b(t) =3$; or
\item[(iii)] $J(0) =1$, ${\rm mult}_0 a(t) =2 $ and ${\rm mult}_0 b(t) \geq 4$.
\end{enumerate}

For the case (i), \cref{eq.limit} and \cref{eq.limitb} give \(\zeta^2=\zeta^3=1\), hence \(\zeta=1\), a contradiction.

For the  case (ii),  \cref{eq.limitb} gives \(\zeta^3=1\)  and $d=3$.  \cref{eq.ab} implies either  \(a(t)\equiv 0\) or
$$ {\rm mult}_0 a(t) \equiv  4 \mod d=3.$$   Thus $ {\rm mult}_0 a(t)$ is of the form
 \(3k+1\) for \(k\geq 1\).  By  Lemma \ref{l.KS} (i), either \(J(t)\equiv 0\) or \({\rm mult}_0 J(t) = 9k-3\) for \(k\geq 1\).
This case is  (8).

For the case (iii), \cref{eq.limit} gives  \(\zeta^2=1\) and $d=2$. \cref{eq.ab} implies either  \(b(t)\equiv 0\) or $$ {\rm mult}_0 b(t) \equiv  6 \mod d=2.$$ Thus $ {\rm mult}_0 b(t)$ is  of the form \(2k+2\) for \(k\geq 1\). By  Lemma \ref{l.KS} (i),   \(J(t)\equiv 1\) or \({\rm mult}_0 J(t) = 4k-2\) for \(k\geq 1\).
This case is (9).

\textbf{The case when \(S_0\) is of type \(\textup{I}_N^*\) with \(N\geq 1\).}
Then \(J(0)=\infty\), $a(0) = b(0) =0$, ${\rm mult}_0 a(t) =2$ and ${\rm mult}_0 b(t) =3$. \cref{eq.limit} and \cref{eq.limitb} imply \(\zeta^2=\zeta^3=1\), hence \(\zeta=1\), a contradiction.

\textbf{The case when \(S_0\) is of type \(\textup{II}\).}
Then $J(0) = 0, a(0) = b(0) =0$,  \({\rm mult}_0 a(t) \geq 1\) and ${\rm mult}_0 b(t)=1$. \cref{eq.limitb} gives \(\zeta^5=1\) and $d=5$. \cref{eq.ab} implies  either \(a(t)\equiv 0\) or $$ {\rm mult}_0 a(t) \equiv  4 \mod d=5,$$ hence $$ {\rm mult}_0 a(t) = 5k-1 \mbox{ for some } k\geq 1.$$ By  Lemma \ref{l.KS} (i), either \(J\equiv 0\) or ${\rm mult}_0 J(t) = 15k-5$ for some \(k\geq 1\).
This case is   (10).

\textbf{The case when \(S_0\) is of type \(\textup{III}\).}
Then $J(0) =1$, $a(0) = b(0) =0$, ${\rm mult}_0 a(t) =1$ and ${\rm mult}_0 b(t) \geq 2.$
\cref{eq.limit} gives  \(\zeta^3=1\) and $d=3$ \cref{eq.ab} implies either \(b(t)\equiv 0\) or $${\rm mult}_0 b(t) \equiv 6 \mod d=3.$$ Hence $$ {\rm mult}_0 b(t) = 3k \mbox{ for some } k\geq 1.$$ By  Lemma \ref{l.KS} (i), either  \(J\equiv 1\) or ${\rm mult}_0 J(t) = 6k-3$ for some \(k\geq 1\).
This case is  (11).

\textbf{The case when \(S_0\) is of type \(\textup{IV}\).}
Then $J(0) = 0$, $a(0) = b(0) =0$, ${\rm mult}_0 a(t) \geq 2$ and ${\rm mult}_0 b(t) =2$.
 \cref{eq.limit} implies \(\zeta^4=1\) and $d=2$ or $4$.
\cref{eq.ab} implies either \(a(t)\equiv 0\) or $${\rm mult}_0 a(t)
\equiv 4 \mod d, $$ hence $$ {\rm mult}_0 a(t) = 2k~(\mbox{if }d=2) \mbox{ or } 4k ~(\mbox{if }d=4)  \mbox{ for some } k\geq 1. $$ By  Lemma \ref{l.KS} (i), \(J(t)\equiv 0\) or ${\rm mult}_0 J(t) = 6k-4$ or $12k-4$ for some  \(k\geq 1\), depending on $d=2$ or $4$, respectively.
This case is   (12) or (13).

\textbf{The case when \(S_0\) is of type \(\textup{II}^*\).}
Then $a(0) = b(0) =0$, \({\rm mult}_0 a(t) \geq 4\) and ${\rm mult}_0 b(t) = 5$.
\cref{eq.limitb} gives \(\zeta=1\), a contradiction.

\textbf{The case when  \(S_0\) is of type \(\textup{III}^*\).}
Then $a(0) = b(0) =0$,  \({\rm mult}_0 a(t) =3\) and ${\rm mult}_0 b(t) \geq 5$.
\cref{eq.limit} gives \(\zeta=1\), a contradiction.

\textbf{The case when \(S_0\) is of type \(\textup{IV}^*\).}
Then $J(0) =0$, $a(0) = b(0) =0$, ${\rm mult}_0 a(t) \geq 3$ and ${\rm mult}_0 b(t) = 4.$
\cref{eq.limitb} gives  \(\zeta^2=1\) and $d=2$. \cref{eq.ab} implies
 either \(a(t)\equiv 0\) or $${\rm mult}_0 a(t) \equiv 4 \mod d=2, \mbox{ hence } {\rm mult}_0 a(t) = 2k+2 \mbox{ for some } k\geq 1.$$ By  Lemma \ref{l.KS} (i), either
\(J(t)\equiv 0\) or  \({\rm mult}_0 J(t) = 6k-2\) for some \(k\geq 1\).
This case is (14).
\end{proof}

Conversely, by choosing suitable Weierstrass coefficients, we can construct symplectic automorphisms for each case of Theorem \ref{t.auto}.

\begin{theorem}\label{t.sigma}
Fix an inhomogeneous coordinate system $(x, y)$ in an affine cell in $\BP^2$. Let $\Delta = \{ t \in \C \mid |t|<1\}$ be the unit disc.
For each case of  Theorem \ref{t.auto}, consider the following pair $(a(t), b(t))$ of polynomial functions on $\Delta$. 
Here we use the convention $t^{\infty} =0$.
\begin{itemize}
\item[(1)] $a(t) = -3(J(0) + t^{2k})(J(0) + t^{2k} -1), b(t) = 2(J(0) + t^{2k})(J(0) + t^{2k} -1)^2$;
\item[(2)] $a(t) = t^{2k}, b(t) \equiv 1$;
\item[(3)] $a(t) = t^{3k-2}, b(t)\equiv 1$;
\item[(4)] $a(t) = t^{6k-2}, b(t) \equiv 1$;
\item[(5)] $a(t) \equiv 1, b(t) = t^{2k}$;
\item[(6)] $a(t) \equiv 1, b(t) = t^{4k-2}$;
\item[(7)]  $a(t) = -3(1-t^{2k}), b(t) = 2(1-t^{2k})^2 $;
\item[(8)] $a(t) = t^{3k+1}, b(t) = t^3$;
\item[(9)] $a(t) = t^2, b(t) = t^{2k+2}$;
\item[(10)] $a(t) = t^{5k-1}, b(t) = t$;
\item[(11)] $a(t) = t, b(t) = t^{3k}$;
\item[(12)] $a(t) = t^{2k}, b(t) = t^2$;
\item[(13)] $a(t) = t^{4k}, b(t) = t^2$;
\item[(14)] $a(t) = t^{2k+2}, b(t) = t^4$.
\end{itemize}
Then the following holds.
 \begin{itemize} \item[(i)] The Weierstrass model of each germ of an elliptic fibration $(f\colon  S \to \Delta, \Sigma)$ in Theorem \ref{t.auto} is biholomorphic to the germ at $t=0$ of the closed surface  $\overline{S} \subset \BP^2 \times \Delta$ defined by the equation $y^2 = x^3 + a(t) x + b(t)$ with the polynomial functions $a(t), b(t)$ listed above.
 \item[(ii)] For the positive integer $d$ in each case of Theorem \ref{t.auto} and $\zeta \in \C$,  a primitive $d$-th root of unity, the holomorphic map  $(t, x, y) \mapsto (\zeta \cdot t, \zeta^2 \cdot x, \zeta^3 \cdot y)$  induces an automorphism $\sigma$ of the germ of the elliptic fibration $(f\colon  S \to \Delta, \Sigma)$ satisfying $\sigma^* \omega = \omega$ for the  symplectic form $\omega $ induced by $\frac{{\rm d} x \wedge {\rm d} t}{y}$ using the notation of  Lemma \ref{l.KS}. \end{itemize}
\end{theorem}

\begin{proof}
For each equivalence class of germs of elliptic fibrations, a Weierstrass equation $y^2 = x^3 + a(t) x + b(t)$ of the equivalence class by
explicit polynomial functions $a(t), b(t)$ are given in \cite[Table VI.1.2]{Mi}.  The polynomials $a(t)$ and $ b(t)$ in Theorem \ref{t.sigma} are copied from there for each case in Theorem \ref{t.auto}. This proves (i).

The holomorphic map in (ii)  transforms the Weierstrass equation $y^2 = x^3 + a(t) x + b(t)$ to
$$ \zeta^6 y^2 = \zeta^6 x^3 + \zeta^2 a(\zeta \cdot t)   x + b(\zeta \cdot t).
$$
So it preserves the Weierstrass model if and only if \begin{equation}\label{eq.uv}
a(\zeta \cdot t)  = \zeta^4 a(t) \ \mbox{ and } \ b(\zeta \cdot t) = \zeta^6 b(t). \end{equation}
It is straightforward to check \cref{eq.uv} for each case in the above list of $a(t)$ and $b(t)$. The equality $\sigma^* \frac{{\rm d} x \wedge {\rm d} t}{y}  = \frac{{\rm d} x \wedge {\rm d} t}{y} $ is  immediate.  Thus the map induces an automorphism $\sigma$ of the germ of the symplectic elliptic fibration $(f\colon  S \to \Delta, \Sigma, \omega)$.
 \end{proof}

From Theorem \ref{t.sigma}, we see that each case of Theorem \ref{t.auto} does occur:

\begin{corollary}\label{c.realize}
For each germ in the list of Theorem \ref{t.auto}, there actually exists a symplectic elliptic fibration with an automorphism of the type described in Theorem \ref{t.auto}, which realizes that germ. \end{corollary}

\section{Symplectic quotient by automorphisms of finite order}\label{s.quotient}
The goal of this section is to prove Theorem \ref{t.isogeny}. First, we need to
state and prove Theorem \ref{t.quotient}, which is necessary for the formulation of Theorem \ref{t.isogeny}. This is a more elaborate description of the local geometry of Proposition \ref{p.descent}.

\begin{theorem}\label{t.quotient}
Let $(f'\colon  S' \to \Delta', \Sigma', \omega')$  be the symplectic elliptic fibration resulting from Proposition \ref{p.descent} applied to the setting of
 Theorem \ref{t.auto}.
Then for each case of Theorem \ref{t.auto}, the equivalence class of the germ of $(f'\colon  S' \to \Delta', \Sigma')$ at $0$ is given by the following triple. We include the value of $d$ for convenience.
\begin{itemize}
\item[(1)] \((\textup{I}^*_{0}, J(0), k), d=2, \) where \(J(0)\neq 0,1,\infty\);
\item[(2)] $(\textup{I}^*_0, 0,  3k), d= 2$;
\item[(3)] $(\textup{IV}^*, 0, 3k-2), d= 3 $;
\item[(4)] $(\textup{II}^*, 0, 3k-1), d=6$;
\item[(5)] $(\textup{I}^*_0, 1,2k), d=2$;
\item[(6)] $(\textup{III}^*, 1, 2k-1), d= 4$;
\item[(7)]  \((\textup{I}^*_{k}, \infty,  k), d= 2\);
\item[(8)] $(\textup{II}^*, 0, 3k-1), d=3$;
\item[(9)] $(\textup{III}^*, 1, 2k-1), d=2$;
\item[(10)] $(\textup{II}^*, 0, 3k-1), d=5$;
\item[(11)] $(\textup{III}^*, 1, 2k-1), d= 3$;
\item[(12)] $(\textup{IV}^*, 0, 3k-2), d=2$;
\item[(13)] $(\textup{II}^*, 0, 3k-1), d=4$;
\item[(14)] $(\textup{II}^*, 0, 3k-1), d=2$.
\end{itemize}

\end{theorem}

\begin{proof}
Let us denote by $0' \in \Delta'$ the origin of $\Delta'$. Then  $J(0')$  is equal to $J(0)$ of $f\colon  S \to \Delta$  at $0 \in \Delta$ and $${\rm mult}_{0'} J = \frac{1}{d} {\rm mult}_0 J.$$ Thus  $(J(0'), {\rm mult}_{0'} J)$ must be as listed in Theorem \ref{t.quotient}. For each case, one can check from \cite[Table VI.1.2]{Mi} or \cite[Table 5.1 in Section 5.9]{SS} that the data $(J(0'), {\rm mult}_{0'} J)$ allow
  only two possibilities for the singularity type of the central fiber. We list them below.
\begin{itemize}
\item[(1)] $\textup{I}_0$ or $\textup{I}^*_0$.
\item[(2)] $\textup{I}_0$ or $\textup{I}^*_0$.
\item[(3)] $\textup{II}$ or $\textup{IV}^*$.
\item[(4)] $\textup{IV}$ or $\textup{II}^*$.
\item[(5)] $\textup{I}_0$ or $\textup{I}_0^*$.
\item[(6)] $\textup{III}$ or $\textup{III}^*$.
\item[(7)] $\textup{I}_k$ or $\textup{I}_k^*$.
\item[(8)]  $\textup{IV}$ or $\textup{II}^*$.
\item[(9)] $\textup{III}$ or $\textup{III}^*$.
\item[(10)] $\textup{IV}$ or $\textup{II}^*$.
\item[(11)] $\textup{III}$ or $\textup{III}^*$.
\item[(12)] $\textup{II}$ or $\textup{IV}^*$.
\item[(13)] $\textup{IV}$ or $\textup{II}^*$.
\item[(14)] $\textup{IV}$ or  $\textup{II}^*$.
\end{itemize}
  It remains to show that only the second type in each case can actually occur.

  First, recall that the homological invariant, namely,  the monodromy on the first homology of a smooth fiber around the origin, is determined by the type of the singular fiber. The monodromy of $f\colon S \to \Delta$ around $0$ must be equal to the $d$-th power of that of $f'\colon  S' \to \Delta'$ around $0'$. From the explicit matrix expressions of the monodromy in \cite[Table 6, Section V.10]{BHPV} or \cite[Table VI.2.1]{Mi}, we can check \begin{itemize}
   \item[(3)] the cubic power of the monodromy of $\textup{II}$ cannot be the monodromy of $\textup{I}_0$;
\item[(8)] the cubic power of the monodromy of $\textup{IV}$ cannot be the monodromy of $\textup{I}^*_0$;
    \item[(10)] the fifth power of the monodromy of  $\textup{IV}$ cannot be the monodromy of $\textup{II}$; and
        \item[(11)] the cubic power of the monodromy of $\textup{III}$ cannot be the monodromy of $\textup{III}$. \end{itemize}
            This settles the above four cases.

  To handle the remaining cases,
  note that the action of $\Z/d \Z$ on $S$ fixes the point $s\coloneqq\Sigma \cap f^{-1}(0)$. Since it preserves $\omega$, the action is trivial on $\wedge^2 T_s S$. Thus the action is locally equivalent at $s$ to the quotient of $\C^2$ by $(z_1, z_2)  \mapsto (\zeta \cdot z_1, \zeta^{-1} \cdot z_2)$,  and  the quotient $\w{S}$ must have a singularity of type $A_{d-1}$. In particular, its minimal resolution has a chain of $d-1$ rational curves, and the central fiber of $S'$ has at least $d$ irreducible components. This shows that only the second possibility is realizable in (1), (2), (4), (5), (6), (12), and (13).
    In the remaining cases of (7), (9), and (14), the order is $d=2$ and the involution of the central fiber of $S$ fixes the point $s$.  From the type of the central fiber of $S$, (7) $\textup{I}_{2k}$, (9) $\textup{I}^*_0$ and (14) $\textup{IV}^*$, it is easy to see that the quotient by an involution has at least $k+1$, $3$ and $5$ components, respectively.
    Thus the central fiber of $S'$ must have at least $k+2$, $4$, and $6$ components, respectively. Thus only the second possibility is realizable in (7), (9), and (14).
    The proof is thus completed.
 \end{proof}

Now we proceed to prove Theorem \ref{t.isogeny}.
The next lemma is well-known (see \cite[Proposition II.1.2]{Mi}).

\begin{lemma}\label{l.Ktrivial}
Let $f\colon  X \to C$ and $\w{f}\colon  \w{X}\to C$ be two relatively minimal genus 1 fibrations over the same Riemann surface $C$. Then a bimeromorphic map $\Phi\colon X \dasharrow\w{X}$ satisfying $f = \w{f} \circ \Phi$ is biholomorphic. \end{lemma}

Recall that an isogeny \((\psi,\Psi)\) from a symplectic elliptic fibration \((\w f\colon \w W\to\w C,\w\Sigma,\w\omega)\) to another symplectic elliptic fibration \((f\colon X\to C,\Sigma,\omega)\) consists of a surjective holomorphic map \(\psi\colon \w C\to C\) and a dominant meromorphic map \(\Psi\colon\w X\dashrightarrow X\) such that over a nonempty Zariski-open subset of \(\w C\), \((\w f\colon \w W\to\w C,\w\Sigma,\w\omega)\) is isomorphic to the base change of \((f\colon X\to C,\Sigma,\omega)\) along \(\psi\) (see Definition \ref{d.isogeny}). 
\begin{lemma}\label{l.isogeny}
With the same notation as above, let $(f'\colon  X' \to \w{C}, \Sigma', \omega')$ be another symplectic elliptic fibration over the Riemann surface $\w{C}$ which admits an isogeny $(\psi, \Psi')$ to $(f\colon  X \to C, \Sigma, \omega)$. Then  $(f'\colon  X' \to \w{C}, \Sigma', \omega')$ is isomorphic to $(\w{f}\colon  \w{X} \to \w{C}, \w{\Sigma}, \w{\omega})$  and there exists a symplectic biholomorphic map \(\Phi\colon X'\to \w X\) such that \(\w f\circ\Phi=f'\). 
In other words, an isogeny $(\psi, \Psi)$ is determined  by $\psi$.   
\end{lemma}

\begin{proof}
By the condition (c) in Definition \ref{d.isogeny} (ii), we have a bimeromorphic map between $\w{X}$ and $X'$ over \(C\), which is compatible with  $\w{f}$ and $f'$.
Thus it must be biholomorphic by Lemma \ref{l.Ktrivial}.
\end{proof}

\begin{proof}[Proof of Theorem \ref{t.isogeny}]
Let us prove (i). For a  point
 $y \in \w{C}$  with ${\rm mult}_y \psi = d >1$, we can choose neighborhoods $y \in \w{O} \subset \w{C}$ and $\psi(y) \in O \subset C$ such that $\psi$ is unramified on  $ \w{O} \setminus \{y \}$ and  $\psi|_{\w{O}}\colon \w{O} \to O$ is isomorphic to the cyclic covering of discs $(t \in \Delta) \mapsto (t^d \in \Delta)$.
 For $0 \neq t$, the composition of two isomorphisms $$\w{X}_t \stackrel{\cong}{\longrightarrow} X_{\psi(t)} \stackrel{\cong}{\longleftarrow} \w{X}_{\zeta \cdot t}$$ induced by $\Psi$ determines a bimeromorphic self-map of $S\coloneqq \w{f}^{-1}(\w{O})$ preserving $\w{f}^{-1}(\w{O}) \cap \w{\Sigma}$. Thus it is an automorphism $\sigma$ of $S$ by Lemma \ref{l.Ktrivial}. It preserves $\w{\omega}$ from the property of $\Psi$. Hence, we are in the setting of Theorem \ref{t.auto}. By Lemma \ref{l.isogeny},  the isogeny $\Psi$ restricted to $S$ must be equal to the one described in Proposition \ref{p.descent} and Theorem \ref{t.quotient}.
 Thus the germ of $(f\colon  X \to C, \Sigma)$ at $\psi(y)$ and $d$ must be one from the list in Theorem \ref{t.quotient}, proving (i).

Now let us prove (ii).
For each ramification point $y \in \w{C}$ of $\psi$, the germ of $(f\colon  X \to C, \Sigma)$ at $\psi(y)$ is one of the list in Theorem \ref{t.quotient}. By Corollary \ref{c.realize},   we can find
\begin{itemize}
\item[(a)] a neighborhood $y \in \w{O} \subset \w{C}$ (resp. $\psi(y) \in O \subset C$) such that  $\psi|_{\w{O}}\colon \w{O} \to O$ is a $d$-cyclic covering ramified only at $y$; \item[(b)]
a   symplectic elliptic fibration $(f_{\w{O}}\colon S \to \w{O}, \Sigma_{\w{O}}, \omega_{\w{O}})$;
\item[(c)] an automorphism  $\sigma$ of order $d$ of the symplectic elliptic fibration $(f_{\w{O}}\colon S \to {\w{O}}, \Sigma_{\w{O}}, \omega_{\w{O}})$; and
\item[(d)] a symplectic form $\omega'$ on $f^{-1}(O)$,
\end{itemize}  such that  the symplectic elliptic fibration $f|_{O}$ is the quotient of $f_{\w{O}}$ by the cyclic group generated by $\sigma$ in the sense of Proposition \ref{p.descent} with an induced isogeny $(\psi|_{\w{O}}\colon \w{O} \to O, \Psi_{\w{O}}\colon S \dasharrow f^{-1}(O))$ from  $(f_{\w{O}}\colon S \to {\w{O}}, \Sigma_{\w{O}}, \omega_{\w{O}})$ to $(f|_{O}\colon  (f)^{-1}(O) \to O, \Sigma \cap f^{-1}(O), \omega')$.
Since $\omega|_{f^{-1}(O)} = g(t) \omega'$ for some nowhere vanishing function $g(t)$ on $O$,  we may replace $\omega_{\w{O}}$  by $\Psi_{\w{O}}^*(g(t) \omega_{\w{O}})$ and  assume that $\omega'= \omega|_{f^{-1}(O)}. $
From the construction in Proposition \ref{p.descent}, we have a natural  isomorphism $\gamma$ of elliptic fibrations over the punctured open set $\w{O} \setminus \{y\}$, $$ (f_{\w{O}}\colon S \to \w{O}, \Sigma_{\w{O}}, \omega_{\w{O}})|_{{\w{O}} \setminus \{y\}} \  \stackrel{\gamma}{\cong} \ (\Psi_{\w{O}}|_{O\setminus \{y\}})^* (f\colon  X \to C, \Sigma, \omega)|_{O' \setminus \{ \psi(y)\}}.$$

Let $\{ y_1, y_2, \ldots \}$ be the set of  ramification points of $\psi$.
Then its complement \(U=\w{C}\setminus\{y_1,y_2,\ldots\}\) is a Zariski-open subset.
Choose a neighborhood  $y_i \in \w{O}_i \subset \w{C}$ for each $i$ satisfying (a) - (d) as above such that $\w{O}_i \cap \w{O}_j = \emptyset$ for $i \neq j$.   Write $(f_{\w{O}}\colon S \to \w{O}, \Sigma_{\w{O}}, \omega_{\w{O}})$, \(\Psi_{\w{O}}\), and $\gamma$ described above for each $\w{O} = \w{O}_i$ as $(f_i\colon S_i \to \w{O}_i, \w{\Sigma}_i, \w{\omega}_i)$, \(\Psi_i\), and $\gamma_i$.

Over the Zariski-open subset \(U\), we have the natural \(\psi\)-pullback
\[
(\psi|_U)^*(f\colon X\to C,\Sigma,\omega)
\]
such that its restriction to each \(\w{O}_i\setminus\{y_i\}\) agrees with \((f_i\colon S_i \to \w{O}_i, \w{\Sigma}_i, \w{\omega}_i)\) by the definition of the isogeny \(\Psi_i\).
Then we patch the pullback  $(\psi|_U)^*(f\colon X\to C,\Sigma,\omega)$  with each $(f_i\colon S_i \to \w{O}_i, \w{\Sigma}_i, \w{\omega}_i)$ for every $i$ via the isomorphism $\gamma_i$ to obtain a symplectic elliptic fibration $(\w{f}\colon  \w{X} \to \w{C}, \w{\Sigma}, \w{\omega})$, which obviously satisfy all the required properties in Definition \ref{d.isogeny}.
\end{proof}

We are ready to prove Theorem \ref{t.generic}.



\begin{proof}[Proof of Theorem  \ref{t.generic}]
It suffices to prove that for any isogeny  $( \psi\colon C\to C', \Psi\colon X\dashrightarrow X')$  from $(f\colon X \to C, \Sigma, \omega)$ to a symplectic elliptic fibration $(f'\colon X' \to C', \Sigma', \omega')$, the map $\psi$ has degree 1.   Since \(X'\) is also an elliptic K3 surface by Lemma \ref{l.K3}, we have $C \cong \BP^1 \cong C'$.

Suppose the contrary that \(\deg\psi>1\).
Then \(\psi\) has at least two branch points.
By Theorem \ref{t.isogeny}, the singular fiber over each branch point is  of the types  listed in Theorem \ref{t.quotient}. In particular, the singular fiber over each branch point has at least five irreducible components. It follows from Shioda–Tate Formula (see \cite[Chapter 11, Corollary 3.4]{Hu})  that the Picard number \(\rho(X')\) is at least 10.  
However, this contradicts
\(\rho(X')=\rho(X)\) by \cite[Theorem 12.9]{SS}.
\end{proof}

\section{Symplectic geometry of isotrivial elliptic fibrations}\label{s.isotrivial}
In this section, we study the symplectic structure of isotrivial elliptic fibrations.
To prove Theorem \ref{t.isotrivial}, it suffices to prove the following.

\begin{theorem}\label{t.shrink}
Let $(f\colon S\to \Delta, \Sigma)$ be an isotrivial elliptic fibration over the unit disc, equipped with two symplectic forms $\omega$ and $\widetilde{\omega}$ on $S$.
Then we can find two neighborhoods $0 \in O \subset \Delta$ and $ 0 \in \widetilde{O} \subset \Delta$ such that $(f\colon  S \to \Delta, \Sigma, \omega)|_O$ and $(f\colon S \to \Delta, \Sigma, \w{\omega})|_{\w{O}}$ are isomorphic symplectic elliptic fibrations. \end{theorem}

We use the following result on the Weierstrass model of an isotrivial elliptic fibration, which can be checked from \cite[Table VI.1.2]{Mi}.

\begin{lemma}\label{l.isotrivial}
In Lemma \ref{l.KS}, assume that $f\colon  S \to \Delta$ is isotrivial and $f^{-1}(0)$ is singular. Then we can choose the equation $y^2 = x^3 + a(t) x + b(t)$ of its Weierstrass model as follows, depending on the constant value of the functional invariant and the singularity type of $f^{-1}(0)$.
\begin{itemize}
\item[(i)] When the functional invariant is identically 0:
$y^2 = x^3 + t^k$ where $k=1,2,3,4$ or $5$, corresponding to the singularity type \(\textup{II}, \textup{IV},\) \(\textup{I}_0^*\), \(\textup{IV}^*\) or \(\textup{II}^*\), respectively.
\item[(ii)] When the functional invariant is identically 1: $y^2 = x^3 + t^k x$ where   $k=1,2$ or $3$, corresponding to the singularity type  $\textup{III}$, \(\textup{I}_0^*\) or \(\textup{III}^*\), respectively.
\item[(iii)] When the functional invariant is a constant different from 0 or 1: $y^2 = x^3 + a t^2 x + b t^3$ for some nonzero constant $a, b \in \C$ with the singularity type \(\textup{I}_0^*\). \end{itemize} \end{lemma}

\begin{proof}[Proof of Theorem \ref{t.shrink}]
First, assume that $f^{-1}(0)$ is smooth. Then we can assume that $S$ is biholomorphic to $E \times \Delta$ for a fixed elliptic curve $E$. Choose a uniformizing coordinate $z$ on $E$. Then we may assume that $\w{\omega} = {\rm d} z \wedge {\rm d} t$ and $\omega = g(t) \w{\omega}$ for a nowhere vanishing holomorphic function $g(t)$ on $\Delta$.
Choose a biholomorphic map $\phi\colon O \to \w{O}$ with $\phi(0) = 0$ such that the holomorphic function $\phi(t)$ in the coordinate $t$ satisfies $\frac{{\rm d} \phi(t)}{{\rm d} t} = g(t)$.
Then $\Phi\colon (t, z) \mapsto (\phi(t), z)$ satisfies $\Phi^* \w{\omega} = \omega$ over $O$.

Now assume that $f^{-1}(0)$ is singular.
The symplectic forms $\omega$ and $\w{\omega}$ induce symplectic forms on the smooth locus of the  Weierstrass model $\overline{f}\colon  \overline{S} \to \Delta$, which we denote by the same symbols $\omega$ and $\w{\omega}$.

By Lemma \ref{l.Ktrivial}, it suffices to find a biholomorphic map  $\phi\colon O \to \w{O}$ and a bimeromorphic map $\Phi\colon (\overline{f})^{-1}(O) \dasharrow (\overline{f})^{-1}(\w{O})$ of the Weierstrass model, which is biholomorphic outside the central fiber $(\overline{f})^{-1}(0),$ such that $$\overline{f} \circ \Phi = \phi \circ \overline{f} \mbox{ and } \Phi^* \w{\omega} = \omega$$
hold on $(\overline{f})^{-1}(O)$.
To do that, we may assume that the Weierstrass equation is of the form listed in
 Lemma \ref{l.isotrivial}. Furthermore, we can assume that  $\widetilde{\omega}$ is the restriction of $\frac{ {\rm d}x \wedge {\rm d} t}{y}$
  from Lemma \ref{l.KS} and $\omega = g(t) \w{\omega}$   for a nowhere vanishing holomorphic function $g(t)$ on $\Delta$.
  Let  $g(t) = \sum_{i=0}^{\infty} g_i t^i, g_0 \neq 0$ be the power series expression in the coordinate $t$.

For each case of Lemma \ref{l.isotrivial}, we construct below a holomorphic map of the form  $$\Phi\colon (t, x, y) \mapsto (\phi(t), u(t) x, v(t) y)$$ for suitable holomorphic functions $\phi(t), u(t), v(t)$ in the coordinate $t$ in a neighborhood of $t=0$ such that
\begin{itemize}
\item the map $\Phi$ preserves the surface defined by the Weierstrass equation, and
\item it satisfies
$$\Phi^* \w{\omega} = \frac{u(t)}{v(t)} \frac{{\rm d} \phi(t)}{{\rm d} t} \frac{{\rm d}x \wedge {\rm d} t}{y}  =  g(t) \w{\omega}  = g(t) \frac{{\rm d}x \wedge {\rm d} t}{y},$$ which is equivalent to
\begin{equation}\label{eq.phi} u(t) \cdot \frac{{\rm d} \phi(t)}{{\rm d} t} = v(t) \cdot g(t).
\end{equation}
\end{itemize} In the argument below, a fractional power of a nowhere vanishing holomorphic function stands for a choice of one of the branches of the fractional power.

\begin{itemize}
\item[(i)]
The case when the Weierstrass equation is $y^2 = x^3 + t^k, 1 \leq k \leq 5$.
Let us choose  $u(t) = (\frac{\phi(t)}{t})^{k/3}$ and $v(t) = (\frac{\phi(t)}{t})^{k/2}$ with
\[
\phi(t)= t\cdot (\frac{6-k}{6})^{\frac{6}{6-k}}\left(\sum_{i=0}^\infty\frac{6}{6i-k+6}g_i t^i\right)^{\frac{6}{6-k}}.
\]
Note that the power series on the right hand side converges in a neighborhood of $0 \in \Delta$ because so does $\sum_{i=0}^{\infty} g_i t^i$. Then $\Phi$ preserves the surface $y^2 = x^3 + t^k$ in $\BP^2 \times \Delta$ and satisfies (\ref{eq.phi}).
\item[(ii)]
The case when the Weierstrass equation is $y^2 = x^3 + t^k x, 1 \leq k \leq 3$.
Let us choose $u (t)=(\frac{\phi(t)}{t})^{k/2}$ and $v(t)=(\frac{\phi(t)}{t})^{3k/4}$ with
$$ \phi(t)=t\cdot (\frac{4-k}{4})^{\frac{4}{4-k}}
\left(\sum_{i=0}^\infty\frac{4}{4i-k+4}g_it^i\right)^{\frac{4}{4-k}}.$$
Note that the power series on the right hand side converges in a neighborhood of $0 \in \Delta$ because so does $\sum_{i=0}^{\infty} g_i t^i$. Then $\Phi$ preserves the surface $y^2 = x^3 + t^k x$ in $\BP^2 \times \Delta$ and satisfies (\ref{eq.phi}).
\item[(iii)]
The case when the Weierstrass equation is $y^2 = x^3 + a t^2 x + b t^3.$
Let us choose $u(t)=\frac{\phi(t)}{t}$ and $v(t)=(\frac{\phi(t)}{t})^{3/2}$ with
$$\phi(t) =\frac{t}{4}\left(\sum_{i=0}^\infty\frac{2}{2i+1}g_it^i\right)^{2}.$$ Note that the power series on the right hand side converges in a neighborhood of $0 \in \Delta$ because so does $\sum_{i=0}^{\infty} g_i t^i$. Then $\Phi$ preserves the surface $y^2 = x^3 + a t^2 x + b t^3$ in $\BP^2 \times \Delta$ and satisfies (\ref{eq.phi}).
\end{itemize}
This finishes the proof of Theorem \ref{t.shrink}.
\end{proof}

\bigskip
Jun-Muk Hwang (jmhwang@ibs.re.kr)

Center for Complex Geometry,
Institute for Basic Science (IBS),
Daejeon 34126, Republic of Korea.

Guolei Zhong (glzhong@math.ecnu.edu.cn, guolei@ibs.re.kr)

Current address: School of Mathematical Sciences, Key Laboratory of MEA(Ministry of Education) \& Shanghai Key Laboratory of PMMP, East China Normal University, Shanghai 200241, China.

Center for Complex Geometry,
Institute for Basic Science (IBS),
Daejeon 34126, Republic of Korea.

\medskip

\end{document}